\documentclass{amsart}

\usepackage{amsmath}             
\usepackage{amsfonts}             
\usepackage{amsthm}               
\usepackage{amsbsy}
\usepackage{amssymb}
\usepackage{amsthm}
\usepackage{amsrefs}

\newtheorem*{theorem*}{Theorem}
\newtheorem*{corollary*}{Corollary}
\newtheorem*{atiyahconjecture*}{Strong Atiyah Conjecture}

\newtheorem{thm}{Theorem}[section]
\newtheorem{lem}[thm]{Lemma}

\newtheorem{cor}[thm]{Corollary}

\theoremstyle{definition}
\newtheorem{defn}[thm]{Definition}

\begin{document}

\nocite{*}

\title[Freely Independent Random Variables with Non-Atomic Distributions]{Freely Independent Random Variables with Non-Atomic Distributions}

\author{Dimitri Shlyakhtenko}
\address{Department of Mathematics, UCLA, Los Angeles, California, 90095, USA}
\email{shlyakht@math.ucla.edu}

\author{Paul Skoufranis}
\address{Department of Mathematics, UCLA, Los Angeles, California, 90095, USA}
\email{pskoufra@math.ucla.edu}

\thanks{This research was supported in part by NSF grants DMS-090076, DMS-1161411 and by NSERC PGS}
\date{\today}

\begin{abstract}
We examine the distributions of non-commutative polynomials of non-atomic, freely independent random variables.  In particular, we obtain an analogue of the Strong Atiyah Conjecture for free groups thus proving that the measure of each atom of any $n \times n$ matricial polynomial of non-atomic, freely independent random variables is an integer multiple of $n^{-1}$.  In addition, we show that the Cauchy transform of the distribution of any matricial polynomial of freely independent semicircular variables is algebraic and thus the polynomial has a distribution that is real-analytic except at a finite number of points.
\end{abstract}

\maketitle

\section{Introduction}
\label{sec:introduction}

One of the essential themes in the study of free probability \cite{VDN} and its applications to random matrix theory is to determine specific properties of the spectral distribution of a fixed (matricial) polynomial in freely independent random variables.  For example, some of the earliest work in free probability theory concerns free convolution, which is the study of the distribution of the polynomial $P(X,Y)=X+Y$ in two freely independent random variables.  In particular, the recent paper \cite{BMS} of Belinschi, Mai, and Speicher uses an analytic theory for operator-valued additive free convolution and Anderson's self-adjoint linearization trick to provide an algorithm for determining distributions of arbitrary polynomials.  Combining the previously known results from \cite{PV2}, \cite{HTS}, \cite{AZ}, and \cite{Sa} along with the results contained in this paper, we obtain the following summary of the known properties of distributions of matrices whose entries are polynomials in several free variables (or, equivalently, polynomials in free variables having matricial coefficients).
\begin{thm}
\label{mainresultinintro}
Let $X_1, \ldots, X_n$ be normal, freely independent random variables and let $[p_{i,j}]$ be an $\ell \times \ell$ matrix whose entries are non-commuting polynomials in $n$ variables and their adjoints such that $[p_{i,j}(X_1, \ldots, X_n)]$ is normal.  Then
\begin{enumerate}
\item \label{intromainresult1} if there exists $\{d_j\}^n_{j=1} \subseteq \mathbb{N}$ such that the measure of each atom in the probability distribution of $X_j$ is an integer multiple of $\frac{1}{d_j}$, then the measure of each atom in the probability distribution of $[p_{i,j}(X_1, \ldots, X_n)]$ is an integer multiple of $\frac{1}{d\ell}$ where $d := \prod^n_{j=1} d_j$.
\end{enumerate}
In particular,
\begin{enumerate}
\setcounter{enumi}{1}
\item \label{intromainresult2} if the probability distribution of each $X_j$ is non-atomic, then the measure of each atom in the probability distribution of $[p_{i,j}(X_1, \ldots, X_n)]$ is an integer multiple of $\frac{1}{\ell}$.
\end{enumerate}
If, in addition, $X_1, \ldots, X_n$ are freely independent semicircular variables or freely independent Haar unitaries and $[p_{i,j}(X_1, \ldots, X_n)]$ is self-adjoint, then
\begin{enumerate}
\setcounter{enumi}{2}
\item \label{intromainresult3} the spectrum of $[p_{i,j}(X_1, \ldots, X_n)]$ is a union of at most $\ell$ disjoint sets each of which is either a closed interval or a point, and
\item \label{intromainresult4} the measure of each connected subset of the spectrum of $[p_{i,j}(X_1, \ldots, X_n)]$ is a multiple of $\frac{1}{\ell}$.
\end{enumerate}
Furthermore, if $\mu$ is the spectral distribution of $[p_{i,j}(X_1, \ldots, X_n)]$, if $K$ is the support of $\mu$, and if $G_\mu$ is the Cauchy transform of $\mu$, then
\begin{enumerate}
\setcounter{enumi}{4}
\item \label{intromainresult5} $G_\mu$ is an algebraic formal power series and thus
\item \label{intromainresult6} there exists a finite subset $A$ of $\mathbb{R}$ such that if $I$ is a connected component of $\mathbb{R} \setminus A$ and $\mu|_{I}$ is the restriction of $\mu$ to $I$, then $\mu|_{I} = 0$ whenever $I \setminus K \neq \emptyset$ and if $I \subseteq K$, then $\mu|_{I}$ has probability density function $Im(g)|_{I}$ where $g$ is an analytic function defined on
\[
W := \{z \in \mathbb{C} \, \mid \, |Im(z)| < \delta\} \setminus \bigcup_{a \in A} \{a - it \, \mid \, t \in [0,\infty)\}
\]
for some $\delta > 0$ such that $g$ agrees with $G_\mu$ on $\{z \in \mathbb{C} \, \mid \, 0 < Im(z) < \delta\}$ and for each $a \in A$ there exists an $N \in \mathbb{N}$ and an $\epsilon > 0$ such that $(z-a)^N g(z)$ admits an expansion on $W \cap \{z \in \mathbb{C} \, \mid \, |z-a| < \epsilon\}$ as a convergent power series in $r_N(z-a)$ where $r_N(z)$ is the analytic $N^{\mbox{th}}$-root of $z$ defined with branch $\mathbb{C} \setminus \{-it \, \mid \, t \in [0,\infty)\}$.
\end{enumerate}
Finally, if the support of $\mu$ is contained in $[0, \infty)$, then
\begin{enumerate}
\setcounter{enumi}{6}
\item \label{intromainresult7} $\lim_{\epsilon \to 0} \int^1_\epsilon \ln(t) \, d\mu(t) > -\infty$.
\end{enumerate}
\end{thm}
In this theorem, by a polynomial in $X_1, \ldots, X_n$ we mean a fixed element of the $*$-algebra generated by $X_1, \ldots, X_n$.
\par
Parts (\ref{intromainresult3}) and (\ref{intromainresult4}) of Theorem \ref{mainresultinintro} follow directly from \cite{PV2}*{Corollary 3.2} which computes the $K$-groups of $C^*_{\textrm{red}}(\mathbb{F}_n)$, the reduced group C$^*$-algebras of the free groups.  The characterization of the $K_0$-group immediate implies that the normalized trace of any projection in $\mathcal{M}_\ell(C^*_{\textrm{red}}(\mathbb{F}_n))$ is an integer multiple of $\ell^{-1}$.  Notice that part (\ref{intromainresult4}) of Theorem \ref{mainresultinintro} does not imply part (\ref{intromainresult2}) of Theorem \ref{mainresultinintro} in the setting of part (\ref{intromainresult4}) as atoms may occur inside a closed interval of the spectrum.  Alternatively, these results were obtained using random matrix techniques in \cite{HTS}.
\par 
Notice that part (\ref{intromainresult2}) of Theorem \ref{mainresultinintro} applies when $X_1,\ldots,X_n$ are freely independent semicircular variables.  Since freely independent semicircular variables describe the non-commutative law of certain independent large random matrices (see \cite{VDN}) we obtain the following application to random matrix theory.  
\par
For each $N \in \mathbb{N}$ let $X_1(N),\ldots, X_n(N)$ be self-adjoint Gaussian random matrices (or, more generally, matrices with independent, identically distributed entries satisfying certain moment conditions; see \cite{VDN} or \cite{HP} for details) and let $p$ be an arbitrary non-constant non-commutative polynomial in $n$ variables which is self-adjoint in the sense that $Y(N) = p(X_1(N),\ldots, X_n(N))$ is always a self-adjoint matrix. Let $\mu_N$ be the empirical spectral measure of $Y(N)$ (that is, $\mu_N[a,b]$ is the average proportion of eigenvalues of $Y(N)$ which lie in $[a,b]$).
\begin{cor}
\label{introcor}
With the notation as above, the measures $\mu_N$ converge to a non-atomic limiting measure $\mu$.
\end{cor}
Indeed, by a result of Voiculescu (see \cite{VDN} or \cite{HP}), it is known that $\mu_N$ converges weakly to a measure $\mu$ that is the law of $p(X_1, \ldots, X_n)$ where $X_1, \ldots, X_n$ are freely independent semicircular variables.  Thus part (\ref{intromainresult2}) of Theorem \ref{mainresultinintro} implies that $\mu$ has no atoms provided $p$ is non-constant.  
\par
The motivation for the proof of Theorem \ref{mainresultinintro} part (\ref{intromainresult2}) stems from the knowledge that the statement of the theorem holds by the Strong Atiyah Conjecture for the free groups in the case when $X_1, \ldots, X_n$ are freely independent Haar unitaries.  The Strong Atiyah Conjecture (motivated by the work in \cite{Atiyah} and proved for a class of groups that includes free groups by Linnell in \cite{Linnell}; also see \cite{Luck} and references therein) states that the kernel projection of an arbitrary matrix with entries taken from the group ring $\mathbb{CF}_n$ of a free group on $n$ generators must have integer von Neumann trace.  To prove our theorem, we consider the analogue of the Strong Atiyah Conjecture for $*$-subalgebras of a tracial von Neumann algebra.  We call this notion the Strong Atiyah Property (since it is known that the Strong Atiyah Conjecture does not hold even for arbitrary group algebras; see \cite{GZ} or \cite{Luck} for example).  It is not hard to see that the Strong Atiyah Property holds for $*$-algebras generated by a single normal element with non-atomic spectral measure.  Our main result states that the Strong Atiyah Property for $*$-algebras is stable under taking free products (in the sense of free probability theory \cite{VDN}) with the group algebra of a free group. Our proof closely follows \cite{Sc} with the main difference of being adapted for free products of algebras and not groups.  Using this result, we are able to conclude that the Strong Atiyah Property holds for any $*$-algebra generated by $X_1,\ldots,X_n$ provided that $X_j$ are free and each has a non-atomic distribution.  
\par 
The proof that part (\ref{intromainresult5}) of Theorem \ref{mainresultinintro} is true in the case $X_1, \ldots, X_n$ are freely independent Haar unitaries is contained in the proof of \cite{Sa}*{Theorem 3.6}.  In Section \ref{sec:CauchyTransformOfSemicircularsIsAlgberaic} we will adapt the proof of \cite{Sa}*{Theorem 3.6} to the semicircular case (see Theorem \ref{cauchytransformofsemicircularsisalgebraic}).  The main idea of the proof is to use the fact that if a certain tracial map on formal power series in a single variable with coefficients in a tracial $*$-algebra $\mathcal{A}$ maps rational formal power series to algebraic formal power series, then the Cauchy transform of a measure associated to a self-adjoint element of $\mathcal{A}$ is algebraic (see Lemma \ref{tracialmaptakesrationaltoalgebraic}).  The proof that the tracial map is as desired in the case $\mathcal{A}$ is generated by semicircular variables follows from demonstrating that a specific formal power series in non-commuting variables is algebraic via a specific property of the semicircular variables (see Lemma \ref{certainformalpowerseriesisalgebraic}).  
\par
It is an interesting question whether the Cauchy transform of any polynomial in freely independent random variable $X_1, \ldots, X_n$ is algebraic provided the Cauchy transform of each $X_j$ is algebraic. 
\par  
The question of whether the Cauchy transform of a measure is an algebraic power series as in part (\ref{intromainresult5}) of Theorem \ref{mainresultinintro} has previously been studied in particular cases.  For example  \cite{RE}*{Example 3.8} demonstrates that the Cauchy transform of the quarter-circular distribution is not algebraic.  Furthermore \cite{RE}*{Corollary 9.5} demonstrates that if $\mu$ and $\nu$ are compactly supported probability measures on $\mathbb{R}$ which have algebraic Cauchy transforms and are the weak limits of the empirical spectral measures of $N \times N$ random matrices, then the free additive convolution $\mu \boxplus \nu$ (see \cite{Voi}) is algebraic.  Moreover, \cite{RE}*{Corollary 9.6} demonstrates that if, in addition, $\mu$ and $\nu$ have support contained in the positive real axis, then the free multiplicative convolution $\mu \boxtimes \nu$ (see \cite{VoMul}) is algebraic.  This question was also considered in \cite{AZ} for limit laws of certain random matrices.  In fact a result much like ours was hinted at in that paper.  Using \cite{AZ}*{Theorem 2.9} we see that part (\ref{intromainresult6}) of Theorem \ref{mainresultinintro} is implied by part (\ref{intromainresult5}) of Theorem \ref{mainresultinintro}.  In particular, part (6) of Theorem \ref{mainresultinintro} directly provides information about the probability density function of $\mu$ by the Stietjes inversion formula.
\par
Finally, in Section \ref{sec:CauchyTransformOfSemicircularsIsAlgberaic}, we will prove part (7) of Theorem \ref{mainresultinintro} by following the proof of \cite{Sa}*{Theorem 3.6} which demonstrates that if the Cauchy transform of a measure is algebraic, then the Novikov-Shubin invariants of the measure are non-zero.  Our interest in part (7) of Theorem \ref{mainresultinintro} comes from the following question: if $p$ is an arbitrary, non-constant, self-adjoint polynomial in $n$ free semicircular variables, must it be the case that the free entropy (as defined in \cite{VoFreeEntopy}) of $p$ is finite?  Indeed elementary arguments may be used to show that if $S$ is a semicircular variable and $p$ is a non-constant polynomial such that $p(S)$ is self-adjoint, then the spectral measure of $p(S)$ has finite free entropy.  Further evidence that this must be true comes from a strengthened version of part (\ref{intromainresult2}) of Theorem \ref{mainresultinintro}  for matrices of the form $[p_{i,j}]$ where each $p_{i,j}\in \textrm{Alg}(S_1,\ldots,S_n)\otimes\textrm{Alg}(S_1,\ldots,S_n)$, which we prove below.  In particular, it follows that the vector of non-commutative difference quotients $JP:=[\partial_1 P,\ldots,\partial_n P]$ (see \cite{Vo5}) has maximal rank whenever $P$ is a non-constant, non-commutative polynomial in $n$ free semicircular variables. 
\par 
Given the success of \cite{BMS} in providing an algorithm for determining the distributions of (matricial) polynomials in semicircular variables, it would also be of interest if an alternate proof of Theorem \ref{mainresultinintro} could be constructed using the ideas and techniques from \cite{BMS}.

\section{The Atiyah Property for Tracial $*$-Algebras}
\label{sec:defnandexam}

In this section we will introduce the notion of the Atiyah Property for tracial $*$-algebra.  In addition, several examples of tracial $*$-algebras that satisfy the Atiyah Property, which will be of use in Section \ref{sec:atiyahconjecturefornoncommutativerandomvariables}, will be provided.
\par 
If $\ell \in \mathbb{N}$ and $\tau$ is a linear functional on an algebra $\mathcal{A}$, then $\tau_\ell$ will denote the linear functional on $\mathcal{M}_\ell(\mathcal{A})$ given by 
\[
\tau_\ell([A_{i,j}]) = \sum^\ell_{i=1} \tau(A_{i,i})
\]
for all $[A_{i,j}] \in \mathcal{M}_\ell(\mathcal{A})$.  Notice that if $\tau$ is tracial (that is, $\tau(AB) = \tau(BA)$ for all $A,B \in \mathcal{A}$), then $\tau_\ell$ is tracial.
\begin{defn}
Let $\mathcal{A}$ be a $*$-subalgebra of $\mathcal{B}(\mathcal{H})$, let $\tau$ be a vector state that is tracial on $\mathcal{A}$, and let $\Gamma$ be an additive subgroup of $\mathbb{R}$ containing $\mathbb{Z}$.  We say that $(\mathcal{A}, \tau)$ has the Atiyah Property with group $\Gamma$ if for any $n,m \in \mathbb{N}$ and $A \in \mathcal{M}_{m,n}(\mathcal{A})$ the kernel of the induced operator $L_A : \mathcal{H}^{\oplus n} \to \mathcal{H}^{\oplus m}$ given by $L_A(\xi) = A\xi$  satisfies $\tau_m(ker(L_A)) \in \Gamma$.  We say that $(\mathcal{A}, \tau)$ has the Strong Atiyah Property if $(\mathcal{A},\tau)$ has the Atiyah Property with group $\mathbb{Z}$.
\end{defn}
Of course the case that $\Gamma = \mathbb{R}$ is of no interest in the above definition.  By the fact that $ker(L_A) = ker(L_{A^*A})$, it suffices to consider $n = m$ in the above definition.  In this case it is easy to see that $ker(L_A) = \overline{Im(L_{A^*})}$ so we may replace kernels with images in the above definition.  Furthermore, if $\mathcal{A}$ is equipped with a C$^*$-norm and $\tau$ is faithful on the C$^*$-algebra generated by $\mathcal{A}$, the tracial representation of $\mathcal{A} \subseteq \mathcal{B}(\mathcal{H})$ clearly does not matter.
\par 
It is clear that if $G$ is a group that satisfies the Strong Atiyah Conjecture (e.g. any free group) and $\tau_G$ is the canonical tracial state on $L(G)$ (the group von Neumann algebra), then $(\mathbb{C}G, \tau_G)$ has the Strong Atiyah Property.  The following provides examples of a tracial $*$-algebras that have the Atiyah Property.  In particular, the following result implies that the tracial $*$-algebra generated by a single semicircular variable has the Strong Atiyah Property with respect to the canonical tracial state (see \cite{VDN} or \cite{HP}).
\begin{lem}
\label{singlemeasuresatsifiesatiyahwithgroupgeneratedbyatoms}
Let $\mu$ be a compactly supported probability measure on $\mathbb{C}$.  Let $\Gamma$ be the topological closure of the additive subgroup of $\mathbb{R}$ generated by $1$ and the measures of the atoms of $\mu$ and let $(\mathcal{A},\tau)$ be the tracial $*$-subalgebra of $L_\infty(\mu) \subseteq \mathcal{B}(L_2(\mu))$ generated by multiplication by polynomials with trace 
\[
\tau(M_p) = \int_{\mathbb{C}} p \, d\mu.
\]
Then $(\mathcal{A},\tau)$ has the Atiyah Property with group $\Gamma$.
\end{lem}
\begin{proof}
Let $\delta_t$ denote the point-mass measure at $t \in \mathbb{C}$.  Then we can write 
\[
\mu = \nu + \sum_{t} \alpha_t \delta_t
\]
where $\nu$ is a non-atomic, compactly supported measure on $\mathbb{C}$ and $\alpha_t \in \Gamma$ for all $t$.  Therefore $\nu(\mathbb{C}) \in \Gamma$ by the construction of $\Gamma$.  
\par
To see that $(\mathcal{A},\tau)$ has the Atiyah Property with group $\Gamma$, fix $\ell \in \mathbb{N}$ and let $[A_{i,j}] \in \mathcal{M}_\ell(\mathcal{A})$.  Viewing each $A_{i,j}$ as a polynomial, we can view $[A_{i,j}]$ as a measureable function from $\mathbb{C}$ to $\mathcal{M}_\ell(\mathbb{C})$.  Moreover, if $P$ is the projection onto the image of $[A_{i,j}]$ (which is in the von Neumann algebra generated by $\mathcal{M}_\ell(\mathcal{A})$ and thus is in $L_\infty(\mu) \overline{\otimes} \mathcal{M}_\ell(\mathbb{C})$) and $P_t \in \mathcal{M}_\ell(\mathbb{C})$ is the projection onto the image of $[A_{i,j}(t)]$, it is elementary to see that $P(t) = P_t$ $\mu$-almost everywhere.  Hence
\[
\tau_\ell(P) = \int_{\mathbb{C}} tr(P(t)) \, d\mu(t) = \int_{\mathbb{C}} \textrm{rank}([A_{i,j}(t)]) \, d\mu(t).
\]
\par 
Recall the rank of a matrix $M \in \mathcal{M}_\ell(\mathbb{C})$ may be obtained by computing the maximum size of a submatrix with non-zero determinant.  However, the pointwise determinant of submatrices of $[A_{i,j}(t)]$ is a polynomial in $t$ and thus is either zero everywhere or non-zero except at a finite number of points.  Hence we obtain that $t \mapsto \textrm{rank}([A_{i,j}(t)])$ is an integer-valued function that is constant except at a finite number of points which may or may not be atoms of $\mu$.  It is then easy to deduce that $\tau_\ell(P)$ is an integer-valued combination of elements of $\Gamma$ and thus lies in $\Gamma$.
\end{proof}
Extending these integration techniques, we obtain the following result for the product of measures on $\mathbb{C}$.  Notice that the tracial $*$-algebra constructed is the tensor product of tracial $*$-algebras from Lemma \ref{singlemeasuresatsifiesatiyahwithgroupgeneratedbyatoms}.
\begin{lem}
\label{productofnonatomicmeasuressatsifiesatiyah}
Let $n \in \mathbb{N}$ and let $\{\mu_j\}^n_{j=1}$ be non-atomic, compactly supported probability measures on $\mathbb{C}$.  Let $\mu$ be the product measure of $\{\mu_j\}^n_{j=1}$ and let $(\mathcal{A},\tau)$ be the tracial $*$-algebra generated by multiplication by the coordinate functions $\{x_j\}^n_{j=1}$ with trace 
\[
\tau(M_f) = \int_{\mathbb{C}^n} f \, d\mu.
\]
Then $(\mathcal{A},\tau)$ has the Strong Atiyah Property . 
\end{lem}
\begin{proof}
We claim that if $p(x_1, \ldots, x_n)$ is a polynomial and $V$ is the zero set of $p(x_1, \ldots, x_n)$, then $\mu(V) \in \{0,1\}$ and $\mu(V) = 1$ only occurs when $p(x_1, \ldots, x_n)$ is the zero polynomial.  To prove this claim, we proceed by induction on $n$ with the case $n = 1$ following from Lemma \ref{singlemeasuresatsifiesatiyahwithgroupgeneratedbyatoms}.  Suppose the claim holds for $n-1$.  Let $p(x_1, \ldots, x_n)$ be any polynomial and let $\nu$ be the product measure of $\{\mu_j\}^{n-1}_{j=1}$.  Clearly the claim is trivial if $p(x_1,\ldots, x_n)$ is the zero polynomial so suppose $p(x_1,\ldots, x_n)$ is not the zero polynomial.  For each $t \in \mathbb{C}$ let
\[
V_t := \{(x_1, \ldots, x_{n-1}) \in \mathbb{C}^n \, \mid \, p(x_1, \ldots, x_{n-1}, t) = 0\}.
\]
Therefore the zero set of $p(x_1,\ldots, x_n)$ is $\bigcup_{t \in\mathbb{C}} V_t$ and $\nu(V_t) \in \{0,1\}$ for each $t \in \mathbb{C}$ by the induction hypothesis.  If $\nu(V_t) = 1$, then $p(x_1, \ldots, x_{n-1}, t)$ must be the zero polynomial which implies $x_n - t$ divides $p(x_1, \ldots, x_n)$ since we can write 
\[
p(x_1, \ldots, x_n) = \sum_{k=1}^{n-1} \sum_{i_k \geq 0} p_{i_1, \ldots, i_{n-1}}(x_n)x_1^{i_1} \cdots x^{i_{n-1}}_{n-1}
\]
where $p_{i_1, \ldots, i_{n-1}}$ are polynomials and if $p_{i_1, \ldots, i_{n-1}}(t) \neq 0$ for at least one $i_1, \ldots, i_{n-1}$, then clearly $p(x_1, \ldots, x_{n-1}, t)$ would not be the zero polynomial.  By degree arguments there are at most a finite number of $t \in \mathbb{C}$ such that $x_n - t$ divides $p(x_1, \ldots, x_n)$ so $\nu(V_t) = 0$ except for a finite number of $t \in \mathbb{C}$.  Since $\mu_n$ contains no atoms, by integrating using Fubini's Theorem we easily obtain that the zero set of $p(x_1,\ldots, x_n)$ has zero $\mu$-measure as desired.
\par
To see that $(\mathcal{A},\tau)$ has the Strong Atiyah Property, fix $\ell \in \mathbb{N}$ and let $[A_{i,j}] \in \mathcal{M}_\ell(\mathcal{A})$.  Thus each $A_{i,j}$ is a multivariable polynomial.  If $P$ is the projection onto the image of $[A_{i,j}]$, then, as in the proof of Lemma \ref{singlemeasuresatsifiesatiyahwithgroupgeneratedbyatoms}, we obtain that
\[
\tau_\ell(P) = \int_{\mathbb{C}^n} \textrm{rank}([A_{i,j}(t_1, \ldots, t_n)]) \, d\mu(t_1, \ldots, t_n).
\]
Since the rank of a matrix can be determined by computing the largest non-zero determinant of a submatrix and since the determinant of any submatrix of $[A_{i,j}(x_1, \ldots, x_n)]$ is a polynomial in $x_1, \ldots, x_n$ whose zero set either has zero or full $\mu$-measure, the result is complete.
\end{proof}  
Next we endeavour to extend the above result to include compactly supported probability measures on $\mathbb{R}$ that have atoms.  We will only focus on measures with atoms that lie in certain subgroups of $\mathbb{Q}$ since the main result of Section \ref{sec:atiyahconjecturefornoncommutativerandomvariables} will also only apply to these groups.
\par 
To discuss such measures, for an additive subgroup $\Gamma$ of $\mathbb{Q}$ and a $d \in \mathbb{N}$ we define
\[
\frac{1}{d} \Gamma := \left\{\frac{1}{d} g \, \mid \, g \in \Gamma\right\},
\]
which is clearly an additive subgroup of $\mathbb{Q}$ that contains $\mathbb{Z}$ if $\Gamma$ contains $\mathbb{Z}$.  As such, the following result is trivial.
\begin{lem}
\label{tensoringwithmatrixalgebranoproblem}
Let $(\mathcal{A}, \tau)$ be a tracial $*$-algebra that has the Atiyah Property with group $\Gamma$ and let $\ell \in \mathbb{N}$.  Then $(\mathcal{M}_\ell(\mathcal{A}), \frac{1}{\ell}\tau_\ell)$ has the Atiyah Property with group $\frac{1}{\ell}\Gamma$.
\end{lem}
\begin{thm}
\label{productofatomicmeasureshasatiyah}
Let $n \in \mathbb{N}$ and let $\{\mu_j\}^n_{j=1}$ be compactly supported probability measures on $\mathbb{C}$.   Let $\mu$ be the product measure of $\{\mu_j\}^n_{j=1}$ and let $(\mathcal{A},\tau)$ be the tracial $*$-algebra generated by multiplication by the coordinate functions $\{x_j\}^n_{j=1}$ with trace 
\[
\tau(M_f) = \int_{\mathbb{C}^n} f \, d\mu.
\]
 Suppose for each $j \in \{1,\ldots, n\}$ there exists a $d_j \in \mathbb{N}$ such that the atoms of $\mu_j$ have measures contained in $\frac{1}{d_j}\mathbb{Z}$.  If $d := \prod^n_{j=1} d_j$, then $(\mathcal{A}, \tau)$ has the Atiyah Property with group $\frac{1}{d}\mathbb{Z}$.
\end{thm}
\begin{proof}
By assumptions, for each $j \in \{1,\ldots, n\}$ we can write 
\[
\mu_j = \mu''_j + \sum_k \frac{\alpha_k}{d_j} \delta_{t_k}
\]
where $\delta_t$ represents the point-mass probability measure at $t$, the sum is finite, $\alpha_k \in \mathbb{N}$, $t_{k_1} \neq t_{k_2}$ if $k_1 \neq k_2$, and $\mu''_j$ is an non-atomic measure.  Notice $\mu''_j(\mathbb{C}) \in \frac{1}{d_j} \mathbb{Z}$.  Let $\mu'_j := \frac{1}{\mu''(\mathbb{C})} \mu''_j$ if $\mu''_j \neq 0$ and let $\mu'_j$ be the Lebesgue measure on $[0,1]$ if $\mu''_j = 0$.  Therefore the tracial $*$-algebra generated by polynomials acting on $L_2(\mu_j)$ can represented a tracial $*$-algebra of diagonal matrices in $\mathcal{M}_{d_j}(\mathcal{B}(L_2(\mu'_j))$ (with respect to the canonical normalized matrix trace) where the polynomial $x$ maps to the matrix with $x$ appearing on the diagonal $d_j\mu''_j(\mathbb{C})$ times and each $t_k$ appearing on the diagonal $\alpha_k$ times.  
\par 
Let $\mu'$ be the product measure of $\{\mu'_j\}^n_{j=1}$ and let $(\mathcal{A}_{\mu'},\tau_{\mu'})$ be the tracial $*$-algebra generated by multiplication by the coordinate functions $\{x_j\}^n_{j=1}$ with trace $\tau_\mu(M_f) = \int_{\mathbb{C}^n} f \, d\mu'$.  By taking tensor products of the tracial $*$-algebras generated by polynomials acting on $L_2(\mu_j)$, it is easily seen using the above representations that $(\mathcal{A},\tau)$ can be represented in the tracial $*$-algebra $(\mathcal{M}_{d}(\mathcal{A}_{\mu'}), \frac{1}{d} (\tau_{\mu'})_{d})$.  Since Lemma \ref{productofnonatomicmeasuressatsifiesatiyah} implies $(\mathcal{A}_{\mu'},\tau_{\mu'})$ has the Strong Atiyah Property, Lemma \ref{tensoringwithmatrixalgebranoproblem} implies $(\mathcal{M}_{d}(\mathcal{A}_{\mu'}), \frac{1}{d} (\tau_{\mu'})_{d})$ has the Atiyah Property with group $\frac{1}{d} \mathbb{Z}$ completing the proof.
\end{proof}

\section{Atiyah Property for Freely Independent Random Variables}
\label{sec:atiyahconjecturefornoncommutativerandomvariables}

The goal of this section is to use the Atiyah Property for tracial $*$-algebras to gain information about the distributions of matricial polynomials of freely independent random variables.  In particular, Theorem \ref{freeproductwithfreegrouptensorwithotherhasatiyah} will enable the extensions of the results from Section \ref{sec:defnandexam} to the non-commutative setting as seen in Theorem \ref{mainapplicationofresult} thus completing the proof of part (1) of Theorem \ref{mainresultinintro}.  The proof of Theorem \ref{freeproductwithfreegrouptensorwithotherhasatiyah}, which is based on the proof of \cite{Sc}*{Proposition 3} (or the updated version \cite{Sc2}*{Proposition 6.1}), will be postponed until the next section in order to focus on the applications of Theorem \ref{freeproductwithfreegrouptensorwithotherhasatiyah}.
\par
Recall that given unital $*$-algebras $\mathcal{A}_i \subseteq \mathcal{B}(\mathcal{H}_i)$ with vector states $\tau_i$ that are tracial on $\mathcal{A}_i$, we can consider the $*$-subalgebra $\mathcal{A}_1 \ast \mathcal{A}_2$ inside the reduced free product C$^*$-algebra $(\mathcal{B}(\mathcal{H}_1),\tau_1) \ast (\mathcal{B}(\mathcal{H}_2),\tau_2)$ generated by $\mathcal{A}_1$ and $\mathcal{A}_2$.  The canonical vector state $\tau_1 \ast \tau_2$ is then a tracial state on $\mathcal{A}_1 \ast \mathcal{A}_2$ (see \cite{VDN} or \cite{HP}).  Similarly we can consider the $*$-subalgebra $\mathcal{A}_1 \odot \mathcal{A}_2$ inside the C$^*$-algebra $\mathcal{B}(\mathcal{H}_1 \otimes \mathcal{H}_2)$ generated by $T \otimes I_{\mathcal{H}_2}$ and $I_{\mathcal{H}_1} \otimes S$ for all $T \in \mathcal{A}_1$ and $S \in \mathcal{A}_2$.  With this notation, it is easy to state the following technical result.
\begin{thm}
\label{freeproductwithfreegrouptensorwithotherhasatiyah}
Let $n \in \mathbb{N}$, let $\mathbb{F}_n$ be the free group on $n$ generators, let $\mathbb{CF}_n$ be the group $*$-algebra equipped with the C$^*$-norm defined by the left regular representation, and let $\tau_{\mathbb{F}_n}$ be the canonical trace on $L(\mathbb{F}_n)$.  Let $\mathcal{A}$ and $\mathcal{B}$ be $*$-subalgebras of the tracial von Neumann algebras with separable preduals $(\mathfrak{M}, \tau_\mathfrak{M})$ and $(\mathfrak{N}, \tau_\mathfrak{N})$ respectively.  Suppose that  $(\mathcal{A} \odot \mathcal{B}, \tau_\mathfrak{M} \overline{\otimes} \tau_\mathfrak{N})$ has the Atiyah Property with group $\frac{1}{d}\mathbb{Z}$ for some $d \in \mathbb{N}$.  Then $((\mathcal{A} \ast \mathbb{C}\mathbb{F}_n) \odot \mathcal{B}, (\tau_\mathfrak{M} \ast \tau_{\mathbb{F}_n}) \overline{\otimes} \tau_\mathfrak{N})$ has the Atiyah Property with group $\frac{1}{d}\mathbb{Z}$.
\end{thm}
Clearly Theorem \ref{freeproductwithfreegrouptensorwithotherhasatiyah} implies the following two results.
\begin{cor}
\label{tensorproductwithFnonbothsides}
If $\mathcal{A}$ and $\mathcal{B}$ are as in Theorem \ref{freeproductwithfreegrouptensorwithotherhasatiyah} and $n,m \in \mathbb{N}$, then $((\mathcal{A} \ast \mathbb{C}\mathbb{F}_n) \odot (\mathcal{B} \ast \mathbb{CF}_m), (\tau_\mathfrak{M} \ast \tau_{\mathbb{F}_n}) \overline{\otimes} (\tau_\mathfrak{N} \ast\tau_{\mathbb{F}_m}))$ has the Atiyah Property with group $\frac{1}{d}\mathbb{Z}$.
\end{cor}
\begin{proof}
This is a simple application of Theorem \ref{freeproductwithfreegrouptensorwithotherhasatiyah} twice using $\mathcal{A} = \mathcal{B}$ and $\mathcal{B} = \mathcal{A} \ast \mathbb{C}\mathbb{F}_n$ the second time.
\end{proof}
\begin{cor}
Let $\mathcal{A}$ be a $*$-subalgebra of a tracial von Neumann algebra with separable predual $(\mathfrak{M}, \tau_\mathfrak{M})$.  Suppose $(\mathcal{A}, \tau_\mathfrak{M})$ has the Atiyah Property with group $\frac{1}{d}\mathbb{Z}$ for some $d \in \mathbb{N}$.  Then $(\mathcal{A} \ast \mathbb{C}\mathbb{F}_n, \tau_\mathfrak{M} \ast \tau_{\mathbb{F}_n})$ has the Atiyah Property with group $\frac{1}{d}\mathbb{Z}$.
\end{cor}
\begin{proof}
Take $\mathcal{B} = \mathbb{C}$ in Theorem \ref{freeproductwithfreegrouptensorwithotherhasatiyah}.
\end{proof}
Using Theorem \ref{freeproductwithfreegrouptensorwithotherhasatiyah} along with the examples of Section \ref{sec:defnandexam}, we obtain the following result which provides important information about the spectral distributions of matricial polynomials of normal, freely independent random variables.
\begin{thm}
\label{mainapplicationofresult}
Let $n \in \mathbb{N}$ and let $X_1, \ldots, X_n$ be normal, freely independent random variables with probability measures $\mu_j$ as distribution respectively.  Suppose for each $j \in \{1,\ldots, n\}$ there exists a $d_j \in \mathbb{N}$ such that the atoms of $\mu_j$ have measures contained in $\frac{1}{d_j}\mathbb{Z}$.  If $\mathcal{A}$ is the unital $*$-algebra generated by $X_1, \ldots, X_n$ (obtained by taking a reduced free product of tracial $*$-algebras), $\tau$ is the canonical trace on $\mathcal{A}$, and $d := \prod^n_{j=1} d_j$, then $(\mathcal{A}, \tau)$ has the Atiyah Property with group $\frac{1}{d}\mathbb{Z}$.
\par
Furthermore, if $[p_{i,j}]$ is an $\ell \times \ell$ matrix whose entries are non-commutative polynomials in $n$ variables and their adjoints such that $[p_{i,j}(X_1, \ldots, X_n)]$ is normal, then the measure of any atom of the spectral distribution of $[p_{i,j}(X_1, \ldots, X_n)]$ with respect to the normalized trace $\frac{1}{\ell}\tau_\ell$ is in $\frac{1}{d\ell}\mathbb{Z}$.
\end{thm}
\begin{proof}
Let $\mu$ be the product measure of $\{\mu_j\}^n_{j=1}$ and let $(\mathcal{A}_0, \tau_0)$ be the tracial $*$-algebra generated by multiplication by the coordinate functions $\{x_j\}^n_{j=1}$ on $L_2(\mu)$ with trace $\tau_0(M_f) = \int_{\mathbb{C}^n} f \, d\mu$.  Clearly each $X_j$ has a representation in $\mathcal{A}_0$ as multiplication by the coordinate function $x_j$ so we will view $X_j \in \mathcal{A}_0$ for all $j \in \{1,\ldots, n\}$.   Let $U := \lambda(1)$ be the canonical generating unitary operator for $L(\mathbb{Z})$.  Then it is easy to see that $X_1$, $UX_2U^*$, $\ldots$, $U^nX_n(U^n)^*$ are freely independent in $\mathcal{A}_0 \ast \mathbb{CZ}$ with respect to the trace $\tau_0 \ast \tau_\mathbb{Z}$.  However, since $(\mathcal{A}_0, \tau_0)$ has the Atiyah Property with group $\frac{1}{d}\mathbb{Z}$ by Theorem \ref{productofatomicmeasureshasatiyah}, $(\mathcal{A}_0 \ast \mathbb{CZ}, \tau_0 \ast \tau_{\mathbb{Z}})$ has the Atiyah Property with group $\frac{1}{d} \mathbb{Z}$ by Theorem \ref{freeproductwithfreegrouptensorwithotherhasatiyah}.  Hence $(\mathcal{A}, \tau)$ has the Atiyah Property with group $\frac{1}{d}\mathbb{Z}$ by taking the canonical isomorphism of tracial $*$-algebras.
\par
Next suppose that $[p_{i,j}]$ is an $\ell \times \ell$ matrix whose entries are non-commutative polynomials in $n$ variables and their adjoints such that $[p_{i,j}(X_1, \ldots, X_n)]$ is normal and the spectral distribution of $[p_{i,j}(X_1, \ldots, X_n)]$ has an atom.  By translation we may assume that this atom occurs at zero and thus corresponds to the kernel projection of $[p_{i,j}(X_1, \ldots, X_n)]$.  Since $(\mathcal{A}, \tau)$ has the Atiyah Property with group $\frac{1}{d}\mathbb{Z}$ we obtain that the measure of the atom is in $\frac{1}{d\ell}\mathbb{Z}$.
\end{proof}
As an application of the above result, we recall that Voiculescu developed in \cite{Voi} the notion of the additive free product of measures in which if $\{X_j\}_{j=1}^n$ are self-adjoint, freely independent random variables with probability measures $\mu_j$ as distribution respectively, then the additive free product measure $\mu := \mu_1 \boxplus \cdots \boxplus \mu_n$ is the distribution of $X_1 + \cdots + X_n$ in the reduced free product C$^*$-algebra.  Hence Theorem \ref{mainapplicationofresult} implies the following specific case of \cite{BV}*{Theorem 7.4}.
\begin{cor}[see \cite{BV}*{Theorem 7.4}]
If $n \in \mathbb{N}$ and $\{\mu_j\}^n_{j=1}$ are non-atomic, compactly supported probability measures on $\mathbb{R}$, then $\mu_1 \boxplus \cdots \boxplus \mu_n$ has no atoms.
\end{cor}
\begin{proof}
Since each $\mu_j$ contains no atoms, we can apply Theorem \ref{mainapplicationofresult} to conclude that $\mu := \mu_1 \boxplus \cdots \boxplus \mu_n$ may only have atoms in $\mathbb{Z}$.  Since $\mu$ is a probability measure, if $\mu$ has an atom, then $\mu$ must be a point-mass measure which would imply that $X_1 + \cdots + X_n = \alpha I$ for some $\alpha \in \mathbb{R}$ contradicting the fact that $X_1$, $\ldots$, $X_n$ are freely independent.
\end{proof}
To complete this section, we can extend Theorem \ref{mainapplicationofresult} to tensor products of tracial $*$-algebras generated by self-adjoint, freely independent random variables.
\begin{cor}
\label{tensorproductcorollary}
Let $n,m \in \mathbb{N}$ and let $X_1, \ldots, X_n$ and $Y_1, \ldots, Y_m$ be collections of normal, freely independent random variables with probability measures $\mu_j$ and $\nu_k$ as distribution respectively.  Let $(\mathcal{A}, \tau_\mathcal{A})$ and $(\mathcal{B}, \tau_\mathcal{B})$ be the tracial $*$-algebras generated by the reduced free products of $\{X_1, \ldots, X_n\}$ and $\{Y_1, \ldots, Y_m\}$ respectively.  Suppose for each $j \in \{1,\ldots, n\}$ and $k \in \{1,\ldots, m\}$ there exists a $d_j, d'_k \in \mathbb{N}$ such that the atoms of $\mu_j$ and $\nu_k$ have measures contained in $\frac{1}{d_j}\mathbb{Z}$ and $\frac{1}{d'_k}\mathbb{Z}$ respectively.  If 
\[
d := \prod^n_{j=1} d_j \cdot \prod^m_{k=1} d'_k,
\]
then $(\mathcal{A} \otimes \mathcal{B}, \tau_\mathcal{A} \overline{\otimes} \tau_\mathcal{B})$ has the Atiyah Property with group $\frac{1}{d}\mathbb{Z}$.
\end{cor}
\begin{proof}
 Let $\mu$ be the product measure of $\{\mu_j\}^n_{j=1}$ and let $\nu$ be the product measure of $\{\nu_k\}^m_{k=1}$.  Let $(\mathcal{A}_0, \tau_{\mathcal{A},0})$ be the tracial $*$-algebra generated by multiplication by the coordinate functions $\{x_j\}^n_{j=1}$ on $L_2(\mu)$ with trace $\tau_{\mathcal{A},0}(M_f) = \int_{\mathbb{C}^n} f \, d\mu$ and let  $(\mathcal{B}_0, \tau_{\mathcal{B},0})$ be the tracial $*$-algebra generated by multiplication by the coordinate functions $\{y_k\}^m_{k=1}$ on $L_2(\nu)$ with trace $\tau_{\mathcal{B},0}(M_f) = \int_{\mathbb{C}^m} f \, d\nu$.  Therefore $(\mathcal{A}_0 \odot \mathcal{B}_0, \tau_{\mathcal{A},0} \overline{\otimes} \tau_{\mathcal{B},0})$ has the Atiyah Property with group $\frac{1}{d}\mathbb{Z}$ by Theorem \ref{productofatomicmeasureshasatiyah}.  The remainder of the proof follows the proof of Theorem \ref{mainapplicationofresult} by an application of Corollary \ref{tensorproductwithFnonbothsides}.
\end{proof}
Notice that Corollary \ref{tensorproductcorollary} has the following interesting application. For any $n,m \in \mathbb{N}$ let $P_1, \ldots, P_m \in \mathcal{A} := \textrm{Alg}(S_1,\ldots,S_n)$ be polynomials in $n$ free semicircular variables $S_1,\ldots,S_n$ and let $\partial_j$ be the non-commutative difference quotient derivations (see \cite{Vo5}).  Let $JP := [\partial_i P_j]_{ij}$ which is an $n\times m$ matrix with entries in $ \mathcal{A}\otimes  \mathcal{A}$.  The matrix $JP$ is the non-commutative Jacobian of $P:=(P_1,\ldots,P_m)$.  We define the rank of $JP$ to be the (non-normalized) trace of its image projection in $\mathcal{M}_n(W^*( \mathcal{A}\otimes  \mathcal{A}))$.  
\begin{cor}
With the above notation, $\textrm{rank}(JP) \in \{0, 1,\ldots,\min(m,n)\}$. In particular, if $\{P_j\}^m_{j=1}$ are not all constant, then $\textrm{rank}(JP)\geq 1$.
\end{cor}

\section{Proof of Theorem \ref{freeproductwithfreegrouptensorwithotherhasatiyah}}
\label{sec:proofofmaintechincallemma}

This section is devoted to the proof of Theorem \ref{freeproductwithfreegrouptensorwithotherhasatiyah}, which underlies all results of Section \ref{sec:atiyahconjecturefornoncommutativerandomvariables}.  Our proof is essentially the same as the argument of Schick in \cite{Sc} adapted for the case of algebras.  This proof has themes similar to those used in \cite{Reich}*{Lemma 10.43}, which makes use of the notion of a Fredholm module to show that the free groups satisfy the Strong Atiyah Conjecture.  The idea of applying Fredholm modules has its roots in a proof of the Kadison Conjecture for free groups on two generators from \cite{Connes}.
\begin{proof}[Proof of Theorem \ref{freeproductwithfreegrouptensorwithotherhasatiyah}]
Let $\mathcal{H} := L_2(\mathfrak{M},\tau_\mathfrak{M})$.  Thus $\mathfrak{M}$ has left and right actions on $\mathcal{H}$.  Similarly, let $\mathcal{K} := L_2(\mathfrak{N},\tau_\mathfrak{N})$.  For a right-$(\mathfrak{M}\overline{\otimes} \mathfrak{N})^{\oplus \ell}$ invariant subspace $\mathcal{L}$ of $(\mathcal{H}  \otimes \mathcal{K})^{\oplus \ell}$, we define
\[
dim_{\mathfrak{M}\overline{\otimes} \mathfrak{N}}(\mathcal{L}) := tr_{\mathfrak{M}\overline{\otimes} \mathfrak{N}}(Q) = (\tau_\mathfrak{M} \overline{\otimes} \tau_\mathfrak{N})_{\ell}(Q)
\]
where $Q$ is the orthogonal projection onto $\mathcal{L}$ (which is an element of $\mathcal{M}_\ell(\mathfrak{M}\overline{\otimes} \mathfrak{N})$ acting on the left).
\par
For later convenience we desire to construct a certain isomorphism of Hilbert spaces that commonly appears in the proof that $\mathbb{F}_n$ satisfies the Strong Atiyah Conjecture. We desire a bijection
\[
\psi : \{\delta_h \, \mid \, h \in \mathbb{F}_n \setminus \{e\}\} \to \{\delta_h \otimes e_i \, \mid \, h \in \mathbb{F}_n, i \in \{1,\ldots, n\}\}.
\]
(where $\{e_i\}^n_{i=1}$ are the canonical orthonormal basis for $\mathbb{C}^n$) as this will clearly produce a unitary operator
\[
\Psi : \ell_2(\mathbb{F}_n) \setminus (\mathbb{C} \delta_e) \to \ell_2(\mathbb{F}_n) \otimes \mathbb{C}^n.
\]
Let $\{u_i\}^n_{i=1}$ be generators for $\mathbb{F}_n$.  Consider the Cayley graph of $\mathbb{F}_n$ with edges $\{g, gu_i\}$.  For each $h \in \mathbb{F}_n \setminus \{e\}$ let $e(h)$ be the first edge of the geodesic from $h$ to $e$.  Thus we may write $e(h) = \{ \psi_0(h), \psi_0(h)u_{r(h)}\}$ for some $r(h) \in \{1,\ldots, n\}$.  Thus if we define
\[
\psi(\delta_h) := \delta_{\psi_0(h)} \otimes e_{r(h)},
\]
we clearly obtain a bijection.
\par 
Let $\lambda$ denote the left regular representation of $\mathbb{F}_n$ on $\ell_2(\mathbb{F}_n)$.  We claim that $\Psi$ has the property that for each $T \in \mathbb{C}\mathbb{F}_n$ the set of $\{\delta_h\}_{h \in \mathbb{F}_n\setminus \{e\}}$ such that $\Psi(\lambda(T)\delta_h)$ does not make sense (i.e. $\langle \lambda(T)\delta_h, \delta_e\rangle \neq 0$) or 
\[
\Psi(\lambda(T)\delta_h) \neq (\lambda(T) \otimes I_{\mathbb{C}^n})\Psi(\delta_h)
\]
is finite.  To see this notice for fixed $g,h \in \mathbb{F}_n$ the only way that $\lambda(g)(\delta_h) \notin \ell_2(\mathbb{F}_n) \ominus (\mathbb{C}\delta_e)$ is if $gh = e$ and the only way that $\Psi(\lambda(g)\delta_h) \neq (\lambda(g) \otimes I_{\mathbb{C}^n})\Psi(\delta_h)$ can occur is if when reducing $gh$ a term from $g$ cancels the second-last letter in $h$ (which occurs for a finite number of $h$ for a given $g$).  Thus the claim follows by the linearity of $\Psi$.
\par 
Let $\{\zeta_j\}_{j \in \mathbb{Z}}$ be any orthonormal basis for $\mathcal{K}$ with $\zeta_0$ a trace vector.  We claim we may assume that there exists an orthonormal basis $\{\xi_j\}_{j \in \mathbb{Z}}$ of $\mathcal{H}$ such that $\xi_0$ is a trace vector and 
\[
\{k \in \mathbb{Z} \, \mid \, \langle T\xi_j, \xi_k\rangle \neq 0\}
\]
is finite for each $j \in \mathbb{Z}$ and $T \in \mathcal{A}$.  To see this, we first may assume that $\mathcal{A}$ is finitely generated by self-adjoint operators $\{A_k\}^m_{k=1}$ since we need only check the Atiyah Property for one matrix with entries in $(\mathcal{A} \ast \mathbb{C}\mathbb{F}_n) \odot \mathcal{B}$ at a time and a finite number of elements of $\mathcal{A}$ will appear.  If $\{\xi'_j\}_{j \in \mathbb{Z}}$ is any orthonormal basis of $\mathcal{H}$ with $\xi'_0 = \xi_0$ a trace vector, then the desired basis will be produced by applying the Gram-Schmidt Orthogonalization Process to 
\[
\{A_{i_1} \cdots A_{i_m}\xi'_j \, \mid \, j \in \mathbb{Z}, m \in \mathbb{N} \cup \{0\}, \{i_k\}^m_{k=1} \subseteq \{1,\ldots, n\}\}
\]
starting with $\xi'_0$.  
\par 
Recall $(\mathcal{A}\ast \mathbb{C}\mathbb{F}_n) \odot \mathcal{B}$ acts on $((\mathcal{H},\xi_0) \ast (\ell_2(\mathbb{F}_n), \delta_e)) \otimes \mathcal{K}$ and 
\[
(\mathcal{H},\xi_0) \ast (\ell_2(\mathbb{F}_n), \delta_e)= \mathbb{C} \xi_0 \oplus \left( \bigoplus \mathbb{C} \left(\xi_{j_1} \otimes \delta_{g_1}\otimes \cdots \right) \right)\oplus \left( \bigoplus \mathbb{C} \left(\delta_{g_1} \otimes \xi_{j_1} \otimes \cdots\right)  \right)
\]
(where $\xi_0 = \delta_e$) where all the tensors in the direct sums have finite length (ending at any point), alternate between basis elements of $\mathcal{H}$ and $\ell_2(\mathbb{F}_n)$, $j_k \in \mathbb{N}$, $i_k \in \mathbb{Z}\setminus \{0\}$, and $g_k \in \mathbb{F}_n \setminus \{e\}$.  Notice that the union of the vectors used in the above definition of $(\mathcal{H},\xi_0) \ast (\ell_2(\mathbb{F}_n), \delta_e)$ is an orthonormal basis for $(\mathcal{H},\xi_0) \ast (\ell_2(\mathbb{F}_n), \delta_e)$.  For convenience of notation, $\xi_0 \otimes \delta_{g_1}  \otimes \cdots := \delta_{g_1} \otimes \cdots$, $\cdots  \otimes \delta_{g_m} \otimes \xi_0 := \cdots \otimes \delta_{g_m}$, and $\cdots \otimes \xi_{j_m} \otimes \delta_e = \cdots \otimes \xi_{j_m}$.
\par
Define the Hilbert spaces
\[
\mathcal{L}_+ := ((\mathcal{H},\xi_0) \ast (\ell_2(\mathbb{F}_n), \delta_e)) \otimes \mathcal{K}
\,\,\mbox{ and }\,\,
 \mathcal{L}_-:= (\mathcal{L}_+ \otimes \mathbb{C}^n \otimes \mathcal{H}) \oplus (\mathcal{H} \otimes \mathcal{K}).
\]
Notice that $(\mathcal{A}\ast \mathbb{C}\mathbb{F}_n) \odot \mathcal{B}$ has a canonical left action on $\mathcal{L}_+$ and thus induces a canonical left action on $\mathcal{L}_-$ by letting an operator $T \in (\mathcal{A}\ast \mathbb{C}\mathbb{F}_n) \odot \mathcal{B}$ act via $(T \otimes I_{\mathbb{C}^n} \otimes I_\mathcal{H}) \oplus 0$.  Thus we may view $\mathcal{L}_+$ and $\mathcal{L}_-$ as left $(\mathcal{A}\ast \mathbb{C}\mathbb{F}_n) \odot \mathcal{B}$-modules.  Similarly, $\mathfrak{M} \overline{\otimes} \mathfrak{N}$ has a canonical right action on $\mathcal{H} \otimes \mathcal{K}$ and thus on $\mathcal{L}_+$ by 
\[
(\cdots \otimes \delta_{g_m} \otimes \xi_{j_m} \otimes \zeta)T = \cdots \otimes \delta_{g_m} \otimes ((\xi_{j_m} \otimes \zeta)T))
\]
for all $\zeta \in \mathcal{K}$.  Hence $\mathcal{L}_+$ is also a right $\mathfrak{M} \overline{\otimes} \mathfrak{N}$-module.  It is clear that the right action of $\mathfrak{M} \overline{\otimes} \mathfrak{N}$ and the left action of $(\mathcal{A} \ast \mathbb{C}\mathbb{F}_n) \odot \mathcal{B}$ on $\mathcal{L}_+$ commute.
\par 
We desire to construct a bijection $\phi$ between the canonical basis elements of $\mathcal{L}_+$ and $\mathcal{L}_-$ which will induce a unitary operator $\Phi : \mathcal{L}_+ \to \mathcal{L}_-$.  It is clear that if $\Lambda := \Lambda_0 \cup \Lambda'$ where $\Lambda_0 = \{\xi_j \otimes \zeta_{j'}\}_{j,j'\in \mathbb{Z}}$ and
\[
\Lambda' := \left\{ (\xi_{j_0} \otimes \delta_{g_1} \otimes \cdots \otimes \delta_{g_m} \otimes \xi_{j_m}) \otimes \zeta_{j'}\, \left| \, 
\begin{array}{cc}
m \geq 1,  \{g_k\}^{m}_{k=1} \in \mathbb{F}_n \setminus \{e\},  \\
j_0, j_m, j' \in \mathbb{Z}, \{j_k\}^{m-1}_{k=1} \subseteq \mathbb{Z}\setminus \{0\}
\end{array}
\right. \right\},
\]
then $\Lambda$ is an orthonormal basis of $\mathcal{L}_+$.  Furthermore 
\[
\Theta := \{0 \oplus (\xi_j \otimes \zeta_{j'})\}_{j,j' \in \mathbb{Z}} \cup \{ (\eta \otimes e_i \otimes \xi_j) \oplus 0\, \mid \, \eta \in \Lambda, j \in \mathbb{Z}, i \in \{1,\ldots, n\}\}
\]
is an orthonormal basis of $\mathcal{L}_-$.  Define $\phi : \Lambda \to \Theta$ by defining $\phi|_{\Lambda_0}$ via
\[
\phi(\xi_j \otimes \zeta_{j'}) = 0 \oplus (\xi_j \otimes \zeta_{j'})
\]
for all $j,j'\in \mathbb{Z}$ and by defining $\phi|_{\Lambda'}$ via the following rule: for 
\[
\eta = (\xi_{j_0} \otimes \delta_{g_1} \otimes \cdots \otimes \delta_{g_m} \otimes \xi_{j_m}) \otimes \zeta_{j'} \in \Lambda
\]
define
\[
\phi(\eta) = (((\xi_{j_0} \otimes \delta_{g_1} \otimes \cdots \otimes \delta_{g_{m-1}} \otimes \xi_{j_{m-1}} \otimes \delta_{\psi_0(g)}) \otimes \zeta_{j'}) \otimes e_{r(g)} \otimes \xi_{j_m}) \oplus 0 
\]
(where if $\psi_0(g) = e$, we reduce the length of the first tensor by removing $\delta_e$).  Since $\Psi$ is a bijection on the given basis elements, it is elementary to verify that $\phi$ is a bijection and thus induces a Hilbert space isomorphism $\Phi : \mathcal{L}_+ \to \mathcal{L}_-$.
\par 
Define a right $\mathfrak{M}\overline{\otimes}\mathfrak{N}$-module structure on $\mathcal{L}_-$ by defining $\eta T := \Phi((\Phi^{-1}(\eta))T)$ for all $T \in \mathfrak{M}\overline{\otimes}\mathfrak{N}$ and $\eta \in \mathcal{L}_-$.  It is easy to see that
\[
(( \eta \otimes e_i \otimes \xi_k) \oplus (\xi_j \otimes \zeta_{j'}))(T \otimes S) = ( \eta(I_\mathcal{H} \otimes S) \otimes e_i \otimes \xi_kT) \oplus (\xi_j T \otimes \zeta_{j'}S)
\]
for all $T \in \mathfrak{M}$ and $S \in \mathfrak{N}$.  Hence $0 \oplus (\mathcal{H} \otimes \mathcal{K})$ and $(\mathcal{L}_+ \otimes \mathbb{C}^n \otimes \mathcal{H}) \oplus 0$ are a right $\mathfrak{M}\overline{\otimes}\mathfrak{N}$-invariant subspace of $\mathcal{L}_-$.  It is clear that the right action of $\mathfrak{M}\overline{\otimes}\mathfrak{N}$ on $\mathcal{L}_-$ commutes with the left action of $(\mathcal{A} \ast \mathbb{C} \mathbb{F}_n) \odot \mathcal{B}$ on $\mathcal{L}_-$.
\par 
Define $\Xi$ to be the union of $\{\xi_0 \otimes \zeta_0\}$ with
\[
\left\{ (\xi_{j_0} \otimes \delta_{g_1} \otimes \cdots \otimes \xi_{j_{m-1}} \otimes \delta_{g_m}) \otimes \zeta_0 \, \left| \, 
\begin{array}{cc}
m \geq 1,  \{g_k\}^{m}_{k=1} \in \mathbb{F}_n \setminus \{e\},  \\
j_0 \in \mathbb{Z}, \{j_k\}^{m-1}_{k=1} \subseteq \mathbb{Z}\setminus \{0\}
\end{array}
\right. \right\}.
\]
It is clear that $\Xi$ is a set of orthonormal vectors in $\mathcal{L}_+$ each of which generates a one-$\mathfrak{M}\overline{\otimes}\mathfrak{N}$-dimensional right $\mathfrak{M}\overline{\otimes}\mathfrak{N}$-submodule of $\mathcal{L}_+$ that are pairwise orthogonal and whose union is dense in $\mathcal{L}_+$ (as $\xi_0$ and $\zeta_0$ are cyclic vectors for the right actions).  By the definition of $\Phi$ it is clear that $\Phi(\Xi)$ is a set of orthonormal vectors in $\mathcal{L}_-$ each of which generates a one-$\mathfrak{M}\overline{\otimes}\mathfrak{N}$-dimensional right $\mathfrak{M}\overline{\otimes}\mathfrak{N}$-submodule of $\mathcal{L}_-$ that are pairwise orthogonal and whose union is dense in $\mathcal{L}_-$.  
\par 
We claim if $T \in (\mathcal{A}\ast \mathbb{C}\mathbb{F}_n) \odot \mathcal{B}$, then 
\[
\{ \xi \in \Xi \, \mid \, \langle T\xi, \xi_j\otimes \zeta_{j'}\rangle_{\mathcal{L}_+} \neq 0 \mbox{ for some }j,j' \in \mathbb{Z} \mbox{ or }\Phi(T(\xi)) \neq T(\Phi(\xi))\}
\]
is a finite subset (containing $\xi_0$).  By linearity it suffices to prove the claim when $T$ is a product of elements from $\mathcal{A} \cup \mathcal{B} \cup \{\lambda(h)\}_{h \in \mathbb{F}_n}$.  First we will prove the claim when $T \in \mathcal{A} \cup \mathcal{B}$. However, it clearly follows that $\langle T\xi, \xi_j \otimes \zeta_{j'}\rangle_{\mathcal{L}_+} \neq 0$ for some $j,j' \in \mathbb{Z}$ or $\Phi(T(\xi)) \neq T\Phi(\xi)$ only if $\xi = \xi_0 \otimes \zeta_0$.
\par
Next we will prove the claim for $T \in \{\lambda(h)\}_{h \in \mathbb{F}_n\setminus \{e\}}$.  Fix $h \in \mathbb{F}_n$, fix $T = \lambda(h)$, and fix 
\[
\xi = \xi_{j_0} \otimes \delta_{g_1} \otimes \cdots \xi_{j_{m-1}}\otimes \delta_{g_m} \otimes \zeta_0 \in \Xi \setminus \{\xi_0 \otimes \zeta_0\}.
\]
If $m > 1$ or $j_0 \neq 0$, then $\langle T\xi, \xi_j  \otimes \zeta_{j'}\rangle = 0$ for all $j, j' \in \mathbb{Z}$ and $\Phi(T(\xi)) = T(\Phi(\xi))$ are clear.  Otherwise $\xi = \delta_{g_1} \otimes \zeta_0$ and it clear that $\langle T\xi, \xi_j\otimes \zeta_{j'}\rangle \neq 0$ for some $j, j' \in \mathbb{Z}$ only if $hg_1 = e$ and $\Phi(T(\xi)) = T(\Phi(\xi))$ unless $\Psi(T \delta_{g_1}) \neq (T \otimes I_{\mathbb{C}^n}) \Psi(\delta_{g_1})$.  Since the number of such $g_1$ is finite, the claim holds in this case.
\par
Next notice for any element $\xi \in \Xi$ and any element $T$ of $\mathcal{A} \cup \{\lambda(h)\}_{h \in \mathbb{F}_n}$ that $T\xi$ is a finite linear combination of elements of $\Xi \cup\{\xi_j\otimes \zeta_0\}_{j \in \mathbb{Z}}$ by the choice of the orthonormal basis $\{\xi_j\}_{j \in \mathbb{Z}}$.  Furthermore, for any element $\xi \in \Xi \cup\{\xi_j\otimes \zeta_0\}_{j \in \mathbb{Z}}$ and any element $T$ of $\mathcal{A} \cup \{\lambda(h)\}_{h \in \mathbb{F}_n}$ there are only a finite number of elements $\eta$ of $\Xi$ such that $\langle T\eta, \xi\rangle_{\mathcal{L}_+} \neq 0$.  Therefore if $T_1, \ldots, T_n \in \mathcal{A} \cup \{\lambda(h)\}_{h \in \mathbb{F}_n}$, then the set of all $\xi \in \Xi$ such that $\langle T_1 \cdots T_n\xi, \xi_j\otimes \zeta_{j'}\rangle_{\mathcal{L}_+} \neq 0$ for some $j,j' \in \mathbb{Z}$, $\langle T_2 \cdots T_n\xi, \xi_j \otimes \zeta_{j'}\rangle_{\mathcal{L}_+} \neq 0$ for some $j,j' \in \mathbb{Z}$, or $\Phi(T_1\cdots T_n\xi) \neq T_1\Phi(T_2 \cdots T_n\xi)$ is finite.  Thus the claim then follows by recursion and the fact that the $\mathcal{B}$-operator commute with elements of $\mathcal{A} \ast \mathbb{CF}_n$ and with $\Phi$.
\par 
The above construction show that we have two representations of $(\mathcal{A} \ast \mathbb{CF}_n) \odot \mathcal{B}$ that differ by a $\mathcal{A} \odot \mathcal{B}$-finite rank operator.  In order to complete the proof, we need a way to analyze the trace of such operators.  Fix $\ell \in \mathbb{N}$ and fix 
\[
A := [A_{i,j}] \in \mathcal{M}_\ell((\mathcal{A}\ast \mathbb{C}\mathbb{F}_n) \odot \mathcal{B}).
\]
The left actions of $(\mathcal{A}\ast \mathbb{C}\mathbb{F}_n) \odot \mathcal{B}$ on $\mathcal{L}_\pm$ allows $A$ to act on $\mathcal{L}_\pm^{\oplus \ell}$.  Let $A_{\pm}$ be the left action of $A$ on $\mathcal{L}_\pm^{\oplus \ell}$ and let $P_\pm \in \mathcal{B}(\mathcal{L}_\pm^{\oplus \ell})$ be the projection onto the image of $A_\pm$.  Thus we desire to show that $((\tau_\mathfrak{M} \ast \tau_{\mathbb{F}_n}) \overline{\otimes} \tau_{\mathfrak{N}})_\ell(P_+) \in \frac{1}{d}\mathbb{Z}$.  Since the right action of $\mathfrak{M}\overline{\otimes}\mathfrak{N}$ on $\mathcal{L}_\pm$ commutes with the left action of $(\mathcal{A}\ast \mathbb{C}\mathbb{F}_n) \odot \mathcal{B}$, we easily obtain that all operators under consideration commute with the diagonal right action of $\mathfrak{M}\overline{\otimes}\mathfrak{N}$ on these spaces.
\par 
Notice that there are only finitely many elements of $(\mathcal{A}\ast \mathbb{C}\mathbb{F}_n) \odot \mathcal{B}$ that appear in $A$.  For each of these elements $T$, we recall that 
\[
\{ \xi \in \Xi \, \mid \, \langle T\xi, \xi_j\otimes \zeta_{j'}\rangle_{\mathcal{L}_+} \neq 0 \mbox{ for some }j,j' \in \mathbb{Z} \mbox{ or }\Phi(T(\xi)) \neq T(\Phi(\xi))\}
\]
is finite.  Let $\mathcal{L}_{+,0}$ be the finite $\mathfrak{M}\overline{\otimes}\mathfrak{N}$-dimensional right $\mathfrak{M}\overline{\otimes}\mathfrak{N}$-submodule of $\mathcal{L}_+$ spanned by the vectors that appear in the above set for at least one $T \in (\mathcal{A}\ast \mathbb{C}\mathbb{F}_n) \odot \mathcal{B}$ appearing in $A$.  Thus $\mathcal{L}_{+,c} := \mathcal{L}_+ \ominus \mathcal{L}_{+,0}$ is a right $\mathfrak{M}\overline{\otimes}\mathfrak{N}$-submodule of $\mathcal{L}_+$.
\par 
Let $\mathcal{L}_{-,c} := \Phi(\mathcal{L}_{+,c})$, which is a right $\mathfrak{M}\overline{\otimes}\mathfrak{N}$-submodule of $\mathcal{L}_-$.  Therefore, since $\mathcal{L}_{+,0}$ contained all $\xi \in \Xi$ where $\Phi(T(\xi)) \neq T(\Phi(\xi))$ for some $T \in (\mathcal{A}\ast \mathbb{C}\mathbb{F}_n) \odot \mathcal{B}$ appearing in $A$ and since the right $\mathfrak{M}\overline{\otimes}\mathfrak{N}$-actions commutes with the left action of $T$ and with $\Phi$, we clearly obtain that
\[
A_+|_{\mathcal{L}_{+,c}} = \Phi^{-1} \circ A_- \circ \Phi|_{\mathcal{L}_{+,c}}.
\]
\par
By progressively adding the right $\mathfrak{M}\overline{\otimes}\mathfrak{N}$-submodule of $\mathcal{L}_+$ generated by a single element of $\Xi$ we can choose an increasing sequence 
\[
\mathcal{L}_{+,0} \subset \mathcal{L}_{+,1} \subset \mathcal{L}_{+,2} \subset \cdots \subset \mathcal{L}_{+}
\]
of finite $\mathfrak{M}\overline{\otimes}\mathfrak{N}$-dimensional right $\mathfrak{M}\overline{\otimes}\mathfrak{N}$-submodules of $\mathcal{L}_{+}$ such that 
\[
\mathcal{L}_{+} = \overline{\bigcup_{j \geq 0} \mathcal{L}_{+,j}}.
\]
Let $\mathcal{L}_{-,j} := \Phi(\mathcal{L}_{+,j})$ for all $j \in \mathbb{N}\cup \{0\}$.  Hence each $\mathcal{L}_{-,j}$ is a right $\mathfrak{M}\overline{\otimes}\mathfrak{N}$-submodule of $\mathcal{L}_{-}$ generated by a finite number of elements of $\Phi(\Xi)$.  Notice that $\Lambda_0 \subseteq \mathcal{L}_{+,0}$ so $0 \oplus (\mathcal{H} \otimes \mathcal{K}) \subseteq \mathcal{L}_{-,0}$.  By construction, it is clear that 
\[
A_\pm (\mathcal{L}_{\pm}^{\oplus \ell}) = \overline{\bigcup_{j\geq 0} A_\pm (\mathcal{L}_{\pm,j}^{\oplus \ell})}.
\]
For each $j \in \mathbb{N} \cup \{0\}$ let $P_{\pm, j}$ be the orthogonal projections onto $\overline{A_\pm (\mathcal{L}_{\pm,j}^{\oplus \ell}})$.
\par
Since only finitely many elements of $(\mathcal{A}\ast \mathbb{C}\mathbb{F}_n) \odot \mathcal{B}$ appear in $A$, by our selection right $\mathfrak{M}\overline{\otimes}\mathfrak{N}$-modules generated by elements of $\Xi$ we see that $A_+$ has finite propagation; that is, for every $j \in \mathbb{N}$ there exists an $n_j \in \mathbb{N}$ such that $A_+ (\mathcal{L}_{+,j}^{\oplus \ell})\subseteq \mathcal{L}_{+,n_j}^{\oplus \ell}$.  Indeed an element of $\mathcal{B}$ does not modify the submodule, $\{\lambda(h)\}_{h \in \mathbb{F}_n}$ permutes the elements of $\Xi$, and an element of $\mathcal{A}$ maps an element of $\Xi$ to at most a finite-$\mathfrak{M}\overline{\otimes}\mathfrak{N}$-dimensional $\mathfrak{M}\overline{\otimes}\mathfrak{N}$-module by the choice of the basis $\{\xi_j\}_{j \in \mathbb{Z}}$.  Similarly, as the left action of $(\mathcal{A}\ast \mathbb{C}\mathbb{F}_n) \odot \mathcal{B}$ on $\mathcal{L}_-$ has the same form and the right $\mathfrak{M}\overline{\otimes}\mathfrak{N}$-modules $\mathcal{L}_{-,j}$ are generated by elements of $\Phi(\Xi)$, $A_-$ also has propagation so we may assume that $A_- (\mathcal{L}_{-,j}^{\oplus \ell}) \subseteq \mathcal{L}_{-,n_j}^{\oplus \ell}$ by choosing $n_j$ sufficiently large.  
\par
The above allows us to view $A_\pm (\mathcal{L}_{\pm,j}^{\oplus \ell})$ as images of rectangular matrices with entries in $\mathcal{A} \odot \mathcal{B}$ acting on the left from $(\mathcal{H}\otimes \mathcal{K})^{\oplus q_j}$ to $(\mathcal{H}\otimes \mathcal{K})^{\oplus p_j}$ for some appropriate choice of $q_j$ and $p_j$.  Indeed an element from $\mathbb{CF}_n$ acting on an element of $\Xi$ or $\Phi(\Xi)$ acts as a scalar matrix since $\{\lambda(h)\}_{h \in \mathbb{F}_n}$ sends the right $\mathfrak{M}\overline{\otimes} \mathfrak{N}$-basis vectors $\Xi$ and $\Phi(\Xi)$ to scalar multiples of other elements of $\Xi$ and $\Phi(\Xi)$ respectively.  Furthermore, each element $T \in \mathcal{A}$ acts by the usual left action of $\mathcal{A}$ on $\mathcal{H} \subseteq \mathcal{L}_{+}$ (which corresponds to the action of $\mathcal{A} \otimes I_\mathcal{K}$ on the right $\mathfrak{M}\overline{\otimes} \mathfrak{N}$-module generated by $\xi_0 \otimes \zeta_0 \in \Xi$) and otherwise act by sending the other elements of $\Xi$ and every element of $\Phi(\Xi)$ to a finite linear combination of elements of $\Xi$ and $\Phi(\Xi)$ respectively and thus can be viewed as scalar matrices on these right $\mathfrak{M}$-modules.  Furthermore, it is clear that an element of $\mathcal{B}$ acts via $I_\mathcal{H} \otimes \mathcal{B}$ on each of the one-$\mathfrak{M} \overline{\otimes} \mathfrak{N}$-dimensional right $\mathfrak{M} \overline{\otimes} \mathfrak{N}$-modules spanned by an element of $\Xi$ or $\Phi(\Xi)$.  Thus the claim follows.  Therefore, since $\mathcal{A}\odot \mathcal{B}$ has the Atiyah Property with group $\frac{1}{d}\mathbb{Z}$, we obtain that
\[
tr_{\mathfrak{M} \overline{\otimes} \mathfrak{N}}(P_{\pm, j}) = dim_{\mathfrak{M}}(A_\pm (\mathcal{L}_{\pm,j}^{\oplus \ell})) \in \frac{1}{d}\mathbb{Z}.
\]
\par
Notice that 
\[
\overline{A_\pm(\mathcal{L}_{\pm,0}^{\oplus \ell})},\,\, \overline{A_\pm(\mathcal{L}_{\pm,c}^{\oplus \ell})},  \mbox{ and each } \overline{A_\pm((\mathcal{L}_{\pm,j} \cap \mathcal{L}_{+,c})^{\oplus \ell})}
\]
are all closed right $\mathfrak{M}\overline{\otimes} \mathfrak{N}$-modules  (note $\mathcal{L}_{\pm,j} \cap \mathcal{L}_{\pm,c} = \mathcal{L}_{\pm,j} \ominus \mathcal{L}_{\pm,0}$).  We claim that 
\[
\begin{array}{l}
dim_{\mathfrak{M}\overline{\otimes} \mathfrak{N}}\left(  \overline{A_\pm(\mathcal{L}_{\pm,0}^{\oplus \ell})\cap \overline{A_\pm(\mathcal{L}_{\pm,c}^{\oplus \ell})} } \right) \\
 = \lim_{j\to \infty} dim_{\mathfrak{M}\overline{\otimes} \mathfrak{N}}\left( \overline{ A_{\pm}(\mathcal{L}_{\pm,0}^{\oplus \ell}) \cap \overline{A_\pm((\mathcal{L}_{\pm,j} \cap \mathcal{L}_{\pm,c})^{\oplus \ell})}}  \right).
 \end{array} 
\]
To see this, it suffices by the continuity of von Neumann dimension (see \cite{Luck}*{proof of Theorem 1.12}) to show that
\[
\overline{A_\pm(\mathcal{L}_{\pm,0}^{\oplus \ell}) \cap \overline{A_\pm(\mathcal{L}_{\pm,c}^{\oplus \ell})}} = \overline{ \bigcup_{j \geq 0} A_{\pm}(\mathcal{L}_{\pm,0}^{\oplus \ell}) \cap \overline{A_\pm((\mathcal{L}_{\pm,j} \cap \mathcal{L}_{\pm,c})^{\oplus \ell})}}.
\]
To see this, notice one inclusion is trivial.  For the other inclusion, recall that $A_\pm$ has finite propagation so there exists an $n_0 \in \mathbb{N}$ such that $A_\pm(\mathcal{L}_{\pm,0}^{\oplus \ell}) \subseteq \mathcal{L}_{\pm,n_0}^{\oplus \ell}$ so
\[
\begin{array}{rcl}
A_\pm(\mathcal{L}_{\pm,0}^{\oplus \ell}) \cap \overline{A_\pm(\mathcal{L}_{\pm,c}^{\oplus \ell})} &=& A_\pm(\mathcal{L}_{\pm,0}^{\oplus \ell}) \cap \mathcal{L}_{\pm,n_0}^{\oplus \ell} \cap \overline{A_\pm(\mathcal{L}_{\pm,c}^{\oplus \ell})} 
  \\
 &=&  A_\pm(\mathcal{L}_{\pm,0}^{\oplus \ell}) \cap \mathcal{L}_{\pm,n_0}^{\oplus \ell} \cap \left(\overline{\bigcup_{j\geq 1}A_\pm((\mathcal{L}_{\pm,j}\cap \mathcal{L}_{\pm,c})^{\oplus \ell})}\right).
 \end{array} 
\]
We claim that
\[
\mathcal{L}_{\pm,n_0}^{\oplus \ell} \cap \left(\overline{\bigcup_{j\geq 1}A_\pm((\mathcal{L}_{\pm,j}\cap \mathcal{L}_{\pm,c})^{\oplus \ell})}\right) = \mathcal{L}_{\pm,n_0}^{\oplus \ell}\cap \overline{A_\pm((\mathcal{L}_{\pm,m}\cap \mathcal{L}_{\pm,c})^{\oplus \ell})}
\]
for some sufficiently large $m \in \mathbb{N}$.  Specifically, to choose $m$, we notice, by the same arguments that $\Phi$ almost commutes with the left actions, that there exists an $m \in \mathbb{N}$ such that if $\eta \in \mathcal{L}_{\pm,m+k} \ominus \mathcal{L}_{\pm,m}$ for any $k\geq 1$, then every entry of $A$ applied to $\eta$ is orthogonal to $\mathcal{L}_{\pm,n_0}$ (that is, there are a finite number of elements $\eta$ of $\Xi$ for which there is an entry $T$ in $A$ such that $T\eta$ has non-zero inner product with an element of $\mathcal{L}_{\pm,n_0} \cap \Xi$).  To see the above equality for this $m \in \mathbb{N}$, we notice that one inclusion is trivial.  For the other inclusion, fix 
\[
\xi \in \mathcal{L}_{\pm,n_0}^{\oplus \ell} \cap \left(\overline{\bigcup_{j\geq 1} A_\pm((\mathcal{L}_{\pm,j}\cap \mathcal{L}_{\pm,c})^{\oplus \ell})}\right).
\]
Thus there exists $\eta_j \in (\mathcal{L}_{\pm,c}\cap \mathcal{L}_{\pm,j})^{\oplus \ell}$ such that $\xi = \lim_{j \to \infty} A_\pm \eta_j$.  Therefore, if $P$ is the projection of $\mathcal{L}_{\pm,c}^{\oplus \ell}$ onto $(\mathcal{L}_{\pm,m} \cap \mathcal{L}_{\pm,c})^{\oplus \ell}$, then
\[
A\eta_j = A(P\eta_j) + \omega_j
\]
where $\omega_j \in (\mathcal{L}_{\pm,n_0}^{\oplus \ell})^\bot$.  Therefore, since 
\[
\lim_{j \to \infty} A_\pm \eta_j = \xi \in \mathcal{L}_{\pm,n_0}^{\oplus \ell},
\]
we obtain that $\lim_{j\to\infty} \omega_j = 0$ and $\xi = \lim_{j\to\infty}  A(P\zeta_j)$ where $P\zeta_j  \in (\mathcal{L}_{\pm,m} \cap \mathcal{L}_{\pm,c})^{\oplus \ell}$ as desired.  Hence the claim is complete.  Thus
\[
\begin{array}{rcl}
A_\pm(\mathcal{L}_{\pm,0}^{\oplus \ell}) \cap \overline{A_\pm(\mathcal{L}_{\pm,c}^{\oplus \ell})}  &=& A_\pm(\mathcal{L}_{\pm,0}^{\oplus \ell}) \cap \mathcal{L}_{\pm,n_0}^{\oplus \ell}\cap \overline{A_\pm((\mathcal{L}_{\pm,m}\cap \mathcal{L}_{\pm,c})^{\oplus \ell})}  \\
 &=&  A_\pm(\mathcal{L}_{\pm,0}^{\oplus \ell}) \cap  \overline{A_\pm((\mathcal{L}_{\pm,m}\cap \mathcal{L}_{\pm,c})^{\oplus \ell})}\\
 &\subseteq&  \overline{ \bigcup_{j \geq 0} A_{\pm}(\mathcal{L}_{\pm,0}^{\oplus \ell}) \cap \overline{A_\pm((\mathcal{L}_{\pm,j} \cap \mathcal{L}_{\pm,c})^{\oplus \ell})}}
 \end{array} 
\]
which completes the claim.  
\par 
Let $P_{\pm, c}$ to be the orthogonal projections onto $\overline{A_\pm (\mathcal{L}_{\pm,c}^{\oplus \ell})}$ and for each $j \in \mathbb{N} \cup \{0\}$ let $P_{\pm, j, c}$ be the orthogonal projection onto $\overline{A_\pm (\mathcal{L}_{\pm,j} \cap \mathcal{L}_{\pm,c})^{\oplus \ell}}$.  Notice that $P_{\pm, c}$ and each $P_{\pm, j, c}$ need not be in the von Neumann algebra generated by $\mathcal{M}_\ell((\mathcal{A}\ast \mathbb{C}\mathbb{F}_n) \odot \mathcal{B})$ but do commute with the right $\mathfrak{M}\overline{\otimes} \mathfrak{N}$-action on their respective spaces.  Since 
\[
A_+|_{\mathcal{L}_{+,c}} = \Phi^{-1} \circ A_- \circ \Phi|_{\mathcal{L}_{+,c}},
\]
we obtain that $P_{+, j, c} = \Phi^{-1} \circ P_{-, j, c} \circ \Phi$ for all $j \in \mathbb{N} \cup \{0\}$ and $P_{+, c} = \Phi^{-1} \circ P_{-, c} \circ \Phi$.  Hence
\[
\langle P_{+, c} \eta, \eta\rangle_{\mathcal{L}_{+}^{\oplus \ell}} = \langle P_{-, c} \Phi(\eta), \Phi(\eta)\rangle_{\mathcal{L}_{-}^{\oplus \ell}} \mbox{ and }\langle P_{+, j, c} \eta, \eta\rangle_{\mathcal{L}_{+}^{\oplus \ell}} = \langle P_{-, j, c} \Phi(\eta), \Phi(\eta)\rangle_{\mathcal{L}_{-}^{\oplus \ell}}
\]
for all $j \in \mathbb{N} \cup \{0\}$ and $\eta \in \mathcal{L}_{+}^{\oplus \ell}$.  
\par 
Let $Q_{\pm} := P_\pm - P_{\pm, c}$ and for each $j \in \mathbb{N} \cup \{0\}$ define $Q_{\pm, j} := P_{\pm, j} - P_{\pm, j, c}$.  Clearly these are projections onto the complements of smaller projections in larger projections.  We claim that 
\[
tr_{\mathfrak{M}\overline{\otimes} \mathfrak{N}}(Q_\pm) = \lim_{j\to\infty} tr_{\mathfrak{M}\overline{\otimes} \mathfrak{N}}(Q_{\pm, j}).
\]
To begin, let $A_0$ denote the restriction of $A_\pm$ to $\mathcal{L}_{\pm,0}^{\oplus \ell}$.  We claim for each fixed $j \in \mathbb{N}$ that
\[
0 \longrightarrow ker(Q_{\pm,j }A_0) \longrightarrow \mathcal{L}_{\pm,0}^{\oplus \ell} \stackrel{Q_{\pm,j} A_0}{\longrightarrow} \overline{Im(Q_{\pm,j})} \longrightarrow 0
\]
is a weakly exact sequence (that is, the images are dense in the kernels).  To see this, it suffices to check weak exactness at $\overline{Im(Q_{\pm,j})}$.  It is clear that $Q_{\pm,j} (A_\pm(\mathcal{L}_{\pm,j}^{\oplus \ell}))$ is dense in $\overline{Im(Q_{\pm,j})}$.  However
\[
A_\pm(\mathcal{L}_{\pm,j}^{\oplus \ell}) = A_\pm(\mathcal{L}_{\pm,0}^{\oplus \ell}) + A_\pm((\mathcal{L}_{\pm,j} \cap \mathcal{L}_{\pm,c})^{\oplus \ell})
\]
and it is clear that $Q_{\pm, j}(A((\mathcal{L}_{\pm,j} \cap \mathcal{L}_{\pm,c})^{\oplus \ell})) = 0$.  Thus $Q_{\pm, j}(A_\pm(\mathcal{L}_{\pm,0}^{\oplus \ell})) = Q_{\pm, j}(A_\pm(\mathcal{L}_{\pm,j}^{\oplus \ell}))$ is dense in $\overline{Im(Q_{\pm, j})}$.  Since each term in the weak exact sequence is a right $\mathfrak{M}\overline{\otimes} \mathfrak{N}$-module and weak exact sequence preserve $\mathfrak{M}\overline{\otimes} \mathfrak{N}$-dimension (see \cite{Luck}*{proof of Theorem 1.12}), we obtain that
\[
dim_{\mathfrak{M}\overline{\otimes} \mathfrak{N}}(\mathcal{L}_{\pm,0}^{\oplus \ell}) = dim_{\mathfrak{M}\overline{\otimes} \mathfrak{N}}(Im(Q_{\pm, j})) + dim_{\mathfrak{M}\overline{\otimes} \mathfrak{N}}(ker(Q_{\pm, j}A_0))
\]
(which are all finite as $dim_{\mathfrak{M}\overline{\otimes} \mathfrak{N}}(\mathcal{L}_{\pm,0}^{\oplus \ell})$ is finite by construction).  Furthermore, it is clear that
\[
ker(Q_{\pm, j}A_0) = \{\eta \in \mathcal{L}_{\pm,0}^{\oplus \ell} \, \mid \, Q_{\pm, j}A_0\eta =0 \}.
\]
Hence the sequence
\[
0 \longrightarrow ker(A_0) \longrightarrow ker(Q_{\pm,j }A_0) \stackrel{A_0}{\longrightarrow} \overline{A_\pm(\mathcal{L}_{\pm,0}^{\oplus \ell}) \cap ker(Q_{\pm, j})} \longrightarrow 0
\]
is weakly exact.  This implies the sequence
\[
0 \longrightarrow ker(A_0) \longrightarrow ker(Q_{\pm,j }A_0) \stackrel{A_0}{\longrightarrow} \overline{A_\pm(\mathcal{L}_{\pm,0}^{\oplus \ell}) \cap \overline{A_\pm((\mathcal{L}_{\pm,j} \cap \mathcal{L}_{\pm,c})^{\oplus \ell})}} \longrightarrow 0
\]
is a weakly exact sequence since it is elementary to verify that 
\[
\overline{A_\pm(\mathcal{L}_{\pm,0}^{\oplus \ell}) \cap ker(Q_{\pm, j})} = \overline{A_\pm(\mathcal{L}_{\pm,0}^{\oplus \ell}) \cap \overline{A_\pm((\mathcal{L}_{\pm,j} \cap \mathcal{L}_{\pm,c})^{\oplus \ell})}}.
\]
Hence we obtain that 
\[
\begin{array}{rl}
&dim_{\mathfrak{M}\overline{\otimes} \mathfrak{N}}(ker(Q_{\pm,j }A_0))  \\
=& dim_{\mathfrak{M}\overline{\otimes} \mathfrak{N}}(ker(A_0)) + dim_{\mathfrak{M}\overline{\otimes} \mathfrak{N}}\left(\overline{A_\pm(\mathcal{L}_{\pm,0}^{\oplus \ell}) \cap \overline{A_\pm((\mathcal{L}_{\pm,j} \cap \mathcal{L}_{\pm,c})^{\oplus \ell})}}\right).
 \end{array} 
\]
By combining the two above dimension equations we obtain that
\[
\begin{array}{rl}
 dim_{\mathfrak{M}\overline{\otimes} \mathfrak{N}}(Im(Q_{\pm, j})) =& dim_{\mathfrak{M}\overline{\otimes} \mathfrak{N}}(\mathcal{L}_{\pm,0}^{\oplus \ell}) - dim_{\mathfrak{M}\overline{\otimes} \mathfrak{N}}(ker(A_0)) \\
& - dim_{\mathfrak{M}\overline{\otimes} \mathfrak{N}}\left(\overline{A_\pm(\mathcal{L}_{\pm,0}^{\oplus \ell}) \cap \overline{A_\pm((\mathcal{L}_{\pm,j} \cap \mathcal{L}_{\pm,c})^{\oplus \ell})}}\right)
 \end{array} 
\]
for each $j \in \mathbb{N}$.  Similarly, by repeating the same arguments we obtain that
\[
\begin{array}{rl}
dim_{\mathfrak{M}\overline{\otimes} \mathfrak{N}}(Im(Q_{\pm}))  
=&dim_{\mathfrak{M}\overline{\otimes} \mathfrak{N}}(\mathcal{L}_{\pm,0}^{\oplus \ell}) - dim_{\mathfrak{M}\overline{\otimes} \mathfrak{N}}(ker(A_0)) \\
& - dim_{\mathfrak{M}\overline{\otimes} \mathfrak{N}}\left(\overline{A_\pm(\mathcal{L}_{\pm,0}^{\oplus \ell}) \cap \overline{A_\pm((\mathcal{L}_{\pm,c})^{\oplus \ell})}}\right).
 \end{array} 
\]
Therefore, as all the terms in the above dimension equations are finite (in fact bounded by $dim_{\mathfrak{M}\overline{\otimes} \mathfrak{N}}(\mathcal{L}_{\pm,0}^{\oplus \ell})$),
\[
\begin{array}{rcl}
tr_{\mathfrak{M}\overline{\otimes} \mathfrak{N}}(Q_\pm) &=&  dim_{\mathfrak{M}\overline{\otimes} \mathfrak{N}}(Im(Q_\pm)) \\
 &=&   \lim_{j \to \infty} dim_{\mathfrak{M}\overline{\otimes} \mathfrak{N}}(Im(Q_{\pm, j})) = \lim_{j\to\infty} tr_{\mathfrak{M}\overline{\otimes} \mathfrak{N}}(Q_{\pm, j}).
 \end{array} 
\]
\par
We will now use $\Xi$ and $\Phi(\Xi)$ to compute traces.  For each $\eta \in \Xi$ and $i \in \{1,\ldots, \ell\}$ let 
\[
\eta_i = (0,0,\ldots, 0, \eta, 0, \ldots, 0) \in \mathcal{L}_{+}^{\oplus \ell}
\]
where $\eta$ is in the $i^{th}$ spot and similarly let 
\[
\phi(\eta_i) = (0, \ldots, 0, \phi(\eta), 0, \ldots, 0) \in \mathcal{L}_{-}^{\oplus \ell}.
\]
Since $\Xi$ and $\Phi(\Xi)$ are orthonormal $\mathfrak{M} \overline{\otimes} \mathfrak{N}$-bases for $\mathcal{L}_{+}$ and $\mathcal{L}_{-}$ respectively, we easily obtain that
\[
tr_{\mathfrak{M} \overline{\otimes} \mathfrak{N}}(Q_+) = \sum_{\eta \in \Xi} \sum^\ell_{i=1} \langle Q_+ \eta_i, \eta_i\rangle_{\mathcal{L}_{+}^{\oplus\ell}}
\]
and
\[
tr_{\mathfrak{M} \overline{\otimes} \mathfrak{N}}(Q_-) = \sum_{\eta \in \Xi} \sum^\ell_{i=1} \langle Q_- \phi( \eta_i), \phi(\eta_i) \rangle_{\mathcal{L}_{-}^{\oplus\ell}}.
\]
Furthermore, we notice if $\eta = \xi_0 \otimes \zeta_0 \in \Xi$, then
\[
\sum^\ell_{i=1} \langle P_+ \eta_i, \eta_i\rangle_{\mathcal{L}_{+}^{\oplus\ell}} = ((\tau_\mathfrak{M} \ast \tau_{\mathbb{F}_n})\overline{\otimes} \tau_\mathfrak{N})_\ell(P_+)
\]
whereas
\[
\sum^\ell_{i=1} \langle P_- \phi(\eta_i),\phi(\eta_i)\rangle_{\mathcal{L}_{-}^{\oplus \ell}} =\sum^\ell_{i=1} 0 = 0
\]
by the definition of $A_-$ and $P_-$.  Finally, we claim that
\[
\sum^\ell_{i=1} \langle P_+ \eta_i, \eta_i\rangle_{\mathcal{L}_{+}^{\oplus\ell}}  - \langle P_- \phi(\eta_i),\phi(\eta_i)\rangle_{\mathcal{L}_{-}^{\oplus \ell}} = 0
\]
for all $\eta \in \Xi \setminus \{\xi_0 \otimes \zeta_0\}$.  To see this, suppose 
\[
\eta = (\xi_{j_0} \otimes \delta_{g_1} \otimes \cdots \otimes \delta_{g_m}) \otimes \zeta_{0} \in \Xi \setminus \{\xi_0 \otimes \zeta_0\}.
\]
Then, by considering the above expression of $\phi(\eta)$ and the right action of $L(\mathbb{F}_n)$ on $(\mathcal{H}, \xi_0) \ast (\ell_2(\mathbb{F}_n), \delta_e)$, there exists a unitary operator $U_\eta \in L(\mathbb{F}_n)$ such that $U_\eta$ commutes with the left actions of $\mathfrak{M}$, $L(\mathbb{F}_n)$, and $\mathfrak{N}$ on $\mathcal{L}_+ \otimes \mathbb{C}^n \otimes \mathcal{H}$ such that $U_\eta \phi(\eta) = \eta \otimes e_{i_0}\otimes \xi_0$ for some $i_0 \in \{1,\ldots, n\}$.  Since every element $T \in (\mathcal{A} \ast \mathbb{CF}_n) \odot \mathcal{B}$ acts on $\mathcal{L}_{-}$ via $(T \otimes I_{\mathbb{C}^n} \otimes I_{\mathcal{H}}) \oplus 0_{\mathcal{H} \otimes \mathcal{K}}$, $P_-$ is $(P_+ \otimes I_{\mathbb{C}^n} \otimes I_\mathcal{H}) \oplus 0_{(\mathcal{H} \otimes \mathcal{K})^{\oplus \ell}}$ so
\[
\begin{array}{rcl}
\sum^\ell_{i=1} \langle P_- \phi(\eta_i),\phi(\eta_i)\rangle_{\mathcal{L}_{-}^{\oplus \ell}}  &=& \sum^\ell_{i=1} \langle P_- U^*_\eta(\eta \otimes e_{i_0} \otimes \xi_0),U^*_\eta(\eta \otimes e_{i_0} \otimes \xi_0) \rangle_{\mathcal{L}_{-}^{\oplus \ell}}  \\
 &=& \sum^\ell_{i=1} \langle P_- (\eta \otimes e_{i_0} \otimes \xi_0),\eta \otimes e_{i_0} \otimes \xi_0 \rangle_{\mathcal{L}_{-}^{\oplus \ell}} \\
 &=&  \sum^\ell_{i=1} \langle P_+ \eta_i, \eta_i\rangle_{\mathcal{L}_{+}^{\oplus\ell}}
 \end{array} 
\]
as claimed.  Hence
\[
\sum_{\eta \in \Xi} \sum^\ell_{i=1} \left( \langle  P_+ \eta_i, \eta_i\rangle_{\mathcal{L}_{+}^{\oplus\ell}} - \langle P_- \phi(\eta_i),\phi(\eta_i)\rangle_{\mathcal{L}_{-}^{\oplus \ell}}  \right) = (\tau \ast \tau_{\mathbb{F}_n})_\ell(P_+).
\]
Thus the proof will be complete if the left-hand side of the above equation is in $\frac{1}{d}\mathbb{Z}$.
\par 
To begin we notice for all $\eta \in \Xi$ and $i \in \{1,\ldots, \ell\}$ that
\[
\begin{array}{rl}
  &\langle  P_+ \eta_i, \eta_i\rangle_{\mathcal{L}_{+}^{\oplus\ell}} - \langle P_- \phi(\eta_i),\phi(\eta_i)\rangle_{\mathcal{L}_{-}^{\oplus \ell}} \\
 = &\langle  P_{+,c} \eta_i, \eta_i\rangle_{\mathcal{L}_{+}^{\oplus\ell}} - \langle P_{-,c} \phi(\eta_i),\phi(\eta_i)\rangle_{\mathcal{L}_{-}^{\oplus \ell}}  \\
 &+ \langle  Q_+ \eta_i, \eta_i\rangle_{\mathcal{L}_{+}^{\oplus\ell}} - \langle Q_- \phi(\eta_i),\phi(\eta_i)\rangle_{\mathcal{L}_{-}^{\oplus \ell}} \\
 =&  0 + \langle  Q_+ \eta_i, \eta_i\rangle_{\mathcal{L}_{+}^{\oplus\ell}} - \langle Q_- \phi(\eta_i),\phi(\eta_i)\rangle_{\mathcal{L}_{-}^{\oplus \ell}}.
 \end{array} 
\]
Similarly, we obtain for all $\eta \in \Xi$, $i \in \{1,\ldots, \ell\}$, and $j \in \mathbb{N}$ that
\[
\begin{array}{rl}
 &\langle  P_{+,j} \eta_i, \eta_i\rangle_{\mathcal{L}_{+}^{\oplus\ell}} - \langle P_{-,j} \phi(\eta_i),\phi(\eta_i)\rangle_{\mathcal{L}_{-}^{\oplus \ell}} \\
 =& \langle  P_{+,c,j} \eta_i, \eta_i\rangle_{\mathcal{L}_{+}^{\oplus\ell}} - \langle P_{-,c,j} \phi(\eta_i),\phi(\eta_i)\rangle_{\mathcal{L}_{-}^{\oplus \ell}}  \\
 &+ \langle  Q_{+,j} \eta_i, \eta_i\rangle_{\mathcal{L}_{+}^{\oplus\ell}} - \langle Q_{-,j} \phi(\eta_i),\phi(\eta_i)\rangle_{\mathcal{L}_{-}^{\oplus \ell}} \\
 =&  0 + \langle  Q_{+,j} \eta_i, \eta_i\rangle_{\mathcal{L}_{+}^{\oplus\ell}} - \langle Q_{-,j} \phi(\eta_i),\phi(\eta_i)\rangle_{\mathcal{L}_{-}^{\oplus \ell}}
 \end{array} 
\]
Since $tr_{\mathfrak{M}\overline{\otimes} \mathfrak{N}}(P_{\pm, j}) = dim_{\mathfrak{M}\overline{\otimes} \mathfrak{N}}(A_\pm (\mathcal{L}_{\pm,j}^{\oplus \ell})) \in  \frac{1}{d}\mathbb{Z}$ for all $j \in \mathbb{N}$, and since $Q_{\pm, j}$ have finite $\mathfrak{M}\overline{\otimes} \mathfrak{N}$-rank (bounded by $dim_{\mathfrak{M}\overline{\otimes} \mathfrak{N}}(\mathcal{L}_{+,0}^{\oplus\ell})$), the following computation is valid:
\[
\begin{array}{l}
tr_{\mathfrak{M}\overline{\otimes} \mathfrak{N}}(Q_{+, j}) - tr_{\mathfrak{M}\overline{\otimes} \mathfrak{N}}(Q_{-, j})   \\
= \sum_{\eta \in \Xi}\sum^\ell_{i=1} \langle  Q_{+,j} \eta_i, \eta_i\rangle_{\mathcal{L}_{+}^{\oplus\ell}} - \langle Q_{-,j} \phi(\eta_i),\phi(\eta_i)\rangle_{\mathcal{L}_{-}^{\oplus \ell}}   \\
= \sum_{\eta \in \Xi}\sum^\ell_{i=1} \langle  P_{+,j} \eta_i, \eta_i\rangle_{\mathcal{L}_{+}^{\oplus\ell}} - \langle P_{-,j} \phi(\eta_i),\phi(\eta_i)\rangle_{\mathcal{L}_{-}^{\oplus \ell}} \\
=  tr_{\mathfrak{M}\overline{\otimes} \mathfrak{N}}(P_{+, j}) - tr_{\mathfrak{M}\overline{\otimes} \mathfrak{N}}(P_{-, j}) \in  \frac{1}{d}\mathbb{Z}.
 \end{array} 
\]
Therefore, since $Q_+$ and $Q_-$ have finite $\mathfrak{M}\overline{\otimes} \mathfrak{N}$-rank (specifically bounded above by $dim_{\mathfrak{M}\overline{\otimes} \mathfrak{N}}(\mathcal{L}_{+,0}^{\oplus\ell})$), we obtain that
\[
tr_{\mathfrak{M}\overline{\otimes} \mathfrak{N}}(Q_+) - tr_{\mathfrak{M}\overline{\otimes} \mathfrak{N}}(Q_-) = \lim_{j \to\infty} tr_{\mathfrak{M}\overline{\otimes} \mathfrak{N}}(Q_{+, j}) - tr_{\mathfrak{M}\overline{\otimes} \mathfrak{N}}(Q_{-, j}) \in \frac{1}{d}\mathbb{Z}.
\]
Hence
\[
\begin{array}{rcl}
((\tau_\mathfrak{M} \ast \tau_{\mathbb{F}_n}) \overline{\otimes} \tau_\mathfrak{N})_\ell(P_+)
&=&  \sum_{\eta \in \Xi} \sum^\ell_{i=1} \left( \langle  P_+ \eta_i, \eta_i\rangle_{\mathcal{L}_{+}^{\oplus\ell}} - \langle P_- \phi(\eta_i),\phi(\eta_i)\rangle_{\mathcal{L}_{-}^{\oplus \ell}}  \right) \\
&=&\sum_{\eta \in \Xi} \sum^\ell_{i=1} \left(\langle  Q_+ \eta_i, \eta_i\rangle_{\mathcal{L}_{+}^{\oplus\ell}} - \langle Q_- \phi(\eta_i),\phi(\eta_i)\rangle_{\mathcal{L}_{-}^{\oplus \ell}}\right) \\
&=& tr_{\mathfrak{M}\overline{\otimes} \mathfrak{N}}(Q_+) - tr_{\mathfrak{M}\overline{\otimes} \mathfrak{N}}(Q_-) \in  \frac{1}{d}\mathbb{Z}
 \end{array} 
\]
which completes the proof.  
\end{proof}

\section{Algebraic Cauchy Transforms of Polynomials in Semicircular Variables}
\label{sec:CauchyTransformOfSemicircularsIsAlgberaic}

In this section we will demonstrate that the Cauchy transform of any self-adjoint matricial polynomial of semicircular variables is algebraic (see Theorem \ref{cauchytransformofsemicircularsisalgebraic}).  Knowing that the Cauchy transform of a measure is algebraic provides information about the spectral distribution of operators as seen in Theorem \ref{mainresultinintro}.  To begin, we recall the notion of a formal power series in commuting variables.
\begin{defn}
Let $n \in \mathbb{N}$ and let $X = \{z_1, \ldots, z_n\}$.  For a ring $R$, a formal power series in commuting variables $X$ with coefficients in $R$ is a map $P : (\mathbb{N} \cup \{0\})^n \to R$ which we will write as
\[
P = \sum_{j=0}^n \sum_{k_j \geq 0} P(k_1, \ldots, k_n) z_1^{k_1} \cdots z_n^{k_n}.
\]
A formal power series $P$ is called a polynomial if $P(k_1, \ldots, k_n) = 0$ except for a finite number of $n$-tuples $(k_1, \ldots, k_n)$.  The set of all formal power series with coefficients in $R$ will be denoted $R[[ X ]]$ and the set of all polynomials with coefficients in $R$ will be denoted $R[ X ]$.  
\end{defn}
The set of formal power series over a ring $R$ can be given a ring structure.  Indeed, if addition on $R[[ X ]]$ is defined coordinate-wise and the product of $P, Q \in R[[X]]$ is defined via the rule
\[
(P+Q)(k_1, \ldots, k_n) = \sum_{j=0}^n \sum^{k_j}_{\ell_j=0} P(\ell_1, \ldots, \ell_n)Q(k_1 - \ell_1, \ldots, k_n - \ell_n),
\]
it is elementary to verify that $R[[ X]]$ is a ring.  Clearly $R[X]$ is a subring of $R[[X]]$ which enables us to construct the quotient field of $R[X]$.  The quotient field of $R[X]$ will be denoted $R(X)$. 
\par 
With the above definitions, we have the following definition essential to this section.
\begin{defn}
Let $n \in \mathbb{N}$, let $X = \{z_1, \ldots, z_n\}$, and let $R$ be an integral domain.  A formal power $P \in R[[X]]$ is said to be algebraic if there exists an $m \in \mathbb{N}$ and $\{q_j\}^m_{j=0} \subseteq R(X)$ not all zero such that
\[
\sum^m_{j=0} q_j P^j = 0.
\]
Equivalently, by clearing denominators, we can require $\{q_j\}^m_{j=0} \subseteq R[X]$.  The set of all algebraic elements of $R[[X]]$ is denoted $R_{\textrm{alg}}[[X]]$.
\end{defn}
Our main interest lies in demonstrating that certain formal power series relating to measures are algebraic.  Thus we recall the following definition.
\begin{defn}
Let $\mu$ be a compactly supported probability measure on $\mathbb{R}$.  The Cauchy transform of $\mu$, denoted $G_\mu$, is the function defined on $\{z \in \mathbb{C} \, \mid \, Im(z) > 0\}$ by
\[
G_\mu(z) = \int_\mathbb{R} \frac{1}{z-t} \, d\mu(t).
\]
\end{defn}
Notice for large enough $z$ it is clear that $G_\mu$ has a Laurent expansion that defines a formal power series in $\mathbb{C}[[ \{\frac{1}{z}\}]]$.  Thus it makes sense to ask whether $G_\mu$ is algebraic.
\par 
In order to state the main result of this section, we will need some additional notation.  Let $\mathfrak{M}$ be a finite von Neumann algebra with a faithful normal tracial state $\tau$.  Let $A \in \mathfrak{M}$ be a fixed self-adjoint operator.   Since $A$ is a self-adjoint element in a von Neumann algebra, for each $t \in \mathbb{R}$ let $E_A(t) \in \mathfrak{M}$ be the spectral projection of $A$ onto $(-\infty,t]$.  The spectral density function of $A$, denoted $F_A$, is the function on $[-\left\|A\right\|,\left\|A\right\|]$ defined by $F_A(t) = \tau(E_A(t))$.  Clearly $F_A$ is a right continuous function that is bounded above by 1.  In turn, $F_A$ defines the spectral measure of $A$, denoted $\mu_A$, by the equation
\[
\mu_A((t_1, t_2]) = F_A(t_2) - F_A(t_1).
\]
Notice that $\mu_A$ is a Borel probability measure supported on $[-\left\|A\right\|,\left\|A\right\|]$.  Recall the spectral measure has the unique property that if $f$ is a continuous function on the spectrum of $A$, then
\[
\tau(f(A)) = \int_0^{\left\|A\right\|} f(t) \, d\mu_A(t).
\]
\par 
With the above notation, we have the following important result which provides information about spectral distributions as indicated in Section \ref{sec:introduction}.
\begin{thm}
\label{cauchytransformofsemicircularsisalgebraic}
Let $n,\ell \in \mathbb{N}$, let $S_1, \ldots, S_n$ be freely independent semicircular variables, let $\mathcal{A}$ be the $*$-algebra generated by $S_1, \ldots, S_n$, and let $A \in \mathcal{M}_\ell(\mathcal{A})$ be a fixed self-adjoint operator.  The Cauchy transform of the spectral measure of $A$ is algebraic.
\end{thm}
In order to prove Theorem \ref{cauchytransformofsemicircularsisalgebraic} we will mimic the proof of \cite{Sa}*{Theorem 3.6} which proves said result when $S_1, \ldots, S_n$ are replaced with freely independent Haar unitaries.  In order to mimic the proof in \cite{Sa}, we recall another type of formal power series in commuting variables.
\begin{defn}
Let $S$ be a ring and let $R$ be a subring of $S$.  It is said that $R$ is rationally closed in $S$ if for every matrix with entries in $R$ which is invertible when viewed as a matrix with entries in $S$, the entries of the inverse lies in $R$.
\par 
The rational closure of $R$ in $S$, denoted $\mathcal{R}(R \subseteq S)$, is the smallest subring of $S$ containing $R$ that is rationally closed. 
\par 
For an arbitrary ring $R$ and finite set $X$, the rational closure $\mathcal{R}(R[ X ] \subseteq R[[X]])$ is called the ring of rational power series over $R$ and is denoted $R_{\textrm{rat}}[[ X ]]$.  
\end{defn}
It turns out that the key to showing the Cauchy transform $G_{\mu_A}$ is algebraic for all positive matrices $A$ with entries in a tracial $*$-algebra is intrinsically related to the following map.
\begin{defn}
Let $\mathfrak{M}$ be a finite von Neumann algebra with faithful, normal, tracial state $\tau$.  The tracial map on formal power series in one variable is the map $Tr_\mathfrak{M} : \mathfrak{M}[[\{z\} ]] \to \mathbb{C}[[\{z\} ]]$ defined by
\[
Tr_\mathfrak{M}\left( \sum_{n\geq 0} T_n z^n\right) = \sum_{n\geq 0} \tau(T_n) z^n.
\]
\end{defn}
In particular, the beginning of the proof of \cite{Sa}*{Theorem 3.6} demonstrates the following.
\begin{lem}
\label{tracialmaptakesrationaltoalgebraic}
Let $\mathfrak{M}$ be a finite von Neumann algebra with faithful, normal, tracial state $\tau$ and let $\mathcal{A}$ be a subalgebra of $\mathfrak{M}$.  If
\[
Tr_\mathfrak{M}(\mathcal{A}_{\textrm{rat}}[[ \{z\} ]]) \subseteq \mathbb{C}_{\textrm{alg}}[[ \{z\} ]],
\]
then the Cauchy transform $G_{\mu_A}$ is algebraic for every positive matrix $A \in \mathcal{M}_\ell(\mathcal{A})$ and any $\ell \in \mathbb{N}$.
\end{lem}
\begin{proof}
As in the proof of \cite{Sa}*{Theorem 3.6}, for an arbitrary $\ell \in \mathbb{N}$ and positive matrix $A \in \mathcal{M}_\ell(\mathcal{A})$, the entries of $z(I_{\mathcal{M}_\ell(\mathcal{A})} - Az)^{-1}$ (which can be viewed as an element of $\mathcal{M}_\ell(\mathcal{A})[[\{z\}]]$ by expanding the result when $\left\|A\right\||z| < 1$) lie in the rational closure $\mathcal{A}_{\textrm{rat}}[[ \{z\}]]$.  By assumption, the formal power series
\[
q(z) := Tr_{\mathcal{M}_\ell(\mathfrak{M})}\left(z(I_{\mathcal{M}_n(\mathcal{A})} - Az)^{-1}\right) = \sum^\ell_{j=1} Tr_\mathfrak{M}((z(I_{\mathcal{M}_\ell(\mathcal{A})} - Az)^{-1})_{jj})
\]
is an element of $\mathbb{C}_{\textrm{alg}}[[ \{z\} ]]$.  Thus $q(z^{-1})$ is an element of $\mathbb{C}_{\textrm{alg}}[[ \{\frac{1}{z}\}]]$.  If $\tau_{\mathcal{M}_\ell(\mathfrak{M})}$ is the canonical trace on $\mathcal{M}_\ell(\mathfrak{M})$, it is well-known that 
\[
G_{\mu_A}(z) = \tau_{\mathcal{M}_\ell(\mathfrak{M})}((zI_{\mathcal{M}_n(\mathcal{A})} - A)^{-1}) = q(z^{-1})
\]
in the domain $\{z \in \mathbb{C} \, \mid \, Im(z) > 0, |z| > \left\|A\right\|\}$.  Hence $G_{\mu_A} \in \mathbb{C}_{\textrm{alg}}[[ \{\frac{1}{z}\}]]$ as desired.
\end{proof}
Thus the proof of Theorem \ref{cauchytransformofsemicircularsisalgebraic} will be complete provided the assumptions of Lemma \ref{tracialmaptakesrationaltoalgebraic} can be verified.  Following \cite{Sa}, it is necessary to examine formal power series in non-commuting variables.
\begin{defn}
Let $X$ be a finite set (which will be called an alphabet) and let $W(X)$ denote the set of all words with letters in $X$.  The empty word will be denoted by $e$.  For a ring $R$, a formal power series with non-commuting variables $X$ with coefficients in $R$ is a map $P : W(X) \to R$ which we will write as
\[
P = \sum_{w \in W(X)} P(w) w.
\]
A formal power series $P$ is called a polynomial $P(w) = 0$ except for a finite number of words $w \in W(X)$.  The set of all formal power series with coefficients in $R$ will be denoted $R\langle\langle X\rangle \rangle$ and the set of all polynomials with coefficients in $R$ will be denoted $R\langle X\rangle$. 
\end{defn}
The set of formal power series over a ring $R$ can be given a ring structure.  Indeed, if addition on $R\langle \langle X\rangle \rangle$ is defined coordinate-wise, and multiplication is defined via the rule
\[
\left( \sum_{w \in W(X)} P(w) w\right) \cdot \left(\sum_{w \in W(X)} Q(w)w\right) = \sum_{w \in W(X)} \left(\sum_{u,v \in W(X), uv = w} P(u)Q(v)\right)w
\]
(notice that for each $w \in W(X)$ there are a finite number of pairs $u,v \in W(X)$ such that $w = uv$), it is elementary to verify that $R\langle \langle X\rangle \rangle$ is a ring.  Thus it makes sense to consider the rational closure of $R\langle X\rangle$ inside $R\langle \langle X\rangle \rangle$ which will be denoted $R_{\textrm{rat}}\langle \langle X\rangle \rangle$.
\par 
As with formal power series in commuting variables, there is a notion of an algebraic formal power series in non-commuting variables.  The definition of such a formal power series is more technical than in the commutative case and is based on the following definition.
\begin{defn}[Sch\"{u}tzenberger]
Let $X := \{x_1,\ldots, x_n\}$ be an alphabet and let $Z := \{z_1,\ldots, z_m\}$ be an alphabet disjoint from $X$.  A proper algebraic system over a ring $R$ is a set of equations $z_i = p_i(x_1, \ldots, x_n, z_1, \ldots, z_m)$ for $i \in \{1,\ldots, m\}$ where each $p_i$ is an element of $R\langle X \cup Z\rangle$ that has no constant term nor term of the form $\alpha z_j$ where $\alpha \in R$ and $j \in \{1,\ldots, n\}$.
\par 
A solution to a proper algebraic system is an $m$-tuple $(P_1, \ldots, P_m) \in R\langle\langle X\rangle \rangle^m$ such that $P_j(e) = 0$ and $p_j(x_1, \ldots, x_n, P_1, \ldots, P_m) = P_j$ for all $j \in \{1,\ldots, m\}$.
\end{defn}
\begin{defn}
A formal power series $P \in R\langle\langle X\rangle \rangle$ is said to be algebraic if $P - P(e) e$ is a component of the solution of a proper algebraic system.  The set all algebraic formal power series in $R\langle\langle X\rangle \rangle$ will be denoted by $R_{\textrm{alg}}\langle\langle X\rangle \rangle$.
\end{defn}
In order to prove the assumptions of Lemma \ref{tracialmaptakesrationaltoalgebraic} hold in the context of Theorem \ref{cauchytransformofsemicircularsisalgebraic}, the proof of \cite{Sa}*{Theorem 2.19(ii)} will be mimicked.  To do so, it is necessary to show that a certain formal power series in non-commuting variables is algebraic.  The following formula involving traces of words of semicircular variables plays a crucial role.
\begin{lem}[See \cite{Vo5}*{Section 3}]
\label{splittingoftraceofwordslemma}
Let $n \in \mathbb{N}$, let $S_1, \ldots, S_n$ be freely independent semicircular variables (with second moments 1), let $\mathcal{A}$ be the $*$-algebra generated by $S_1, \ldots, S_n$, let $\tau$ be the canonical trace on $\mathcal{A}$, and let $X := \{x_1, \ldots, x_n\}$ be an alphabet.  For each $j \in \{1,\ldots, n\}$ and $w \in W(X)$, 
\[
\tau(S_j w(S_1, \ldots, S_n)) = \sum_{u,v \in W(X), w = ux_jv} \tau(u(S_1, \ldots, S_n)) \tau(v(S_1,\ldots, S_n))
\]
where, for a word $w_0 \in W(X)$, $w_0(S_1, \ldots, S_n)$ is the element of $\mathcal{A}$ obtained by substituting $S_i$ for $x_i$.
\end{lem}
\begin{lem}
\label{certainformalpowerseriesisalgebraic}
With the notation as in Lemma \ref{splittingoftraceofwordslemma}, the formal power series $P_{semi} \in \mathbb{C}\langle \langle X\rangle \rangle$ defined by
\[
P_{semi} := \sum_{w \in W(X)} \tau(w(S_1,\ldots, S_n)) w
\]
is algebraic.
\end{lem}
\begin{proof}
By Lemma \ref{splittingoftraceofwordslemma} we easily obtain that
\[
\begin{array}{rcl}
 && P_{semi} - e  \\
 &=& \sum^n_{j=1} \sum_{w \in W(X)} \tau(S_j w(S_1, \ldots, S_n)) x_j w  \\
 &=& \sum^n_{j=1} \sum_{w, u,v \in W(X), w=ux_jv} \tau(u(S_1, \ldots, S_n)) \tau(v(S_1, \ldots, S_n))x_j ux_jv \\
 &=&  \sum^n_{j=1} \sum_{u,v \in W(X)}  \tau(u(S_1, \ldots, S_n)) \tau(v(S_1, \ldots, S_n))x_j ux_jv \\
 &=&  \sum^n_{j=1} x_j P_{semi} x_j P_{semi}.
 \end{array} 
\]
Hence it is elementary to verify that $P_{semi} - e$ is a solution to the proper algebraic system
\[
z = \sum^n_{j=1} x_j z x_j z + x_j^2 z + x_j z x_j + x_j^2.
\]
Thus $P_{semi}$ is algebraic by definition.
\end{proof}
Using Lemma \ref{certainformalpowerseriesisalgebraic} it is easy to verify the proof of \cite{Sa}*{Theorem 2.19(ii)} generalizes enough to complete the proof of Theorem \ref{cauchytransformofsemicircularsisalgebraic}.  We will only sketch the changes to the proof of \cite{Sa}*{Theorem 2.19(ii)} as it nearly follows verbatim.
\begin{proof}[Proof of Theorem \ref{cauchytransformofsemicircularsisalgebraic}]
Let $\mathfrak{M}$ be the von Neumann algebra generated by $S_1, \ldots, S_n$.  By Lemma \ref{tracialmaptakesrationaltoalgebraic} it suffices to show that the tracial map on formal power series $Tr_\mathfrak{M} : \mathfrak{M}[[\{z\} ]] \to \mathbb{C}[[ \{z\} ]]$ has the property that
\[
Tr_\mathfrak{M}(\mathcal{A}_{\textrm{rat}}[[ \{z\} ]]) \subseteq \mathbb{C}_{\textrm{alg}}[[ \{z\} ]].
\]
Let $S := \{x_1,\ldots, x_n\}$ be an alphabet.  As in the proof of \cite{Sa}*{Theorem 2.19(ii)}, there is a canonical way to view
\[
(\mathbb{C}\langle S\rangle)_{\textrm{rat}}[[ \{z\} ]] \subseteq (\mathbb{C}(z))_{\textrm{rat}}\langle\langle S\rangle\rangle .
\]
\par 
Consider the injective homomorphisms $\pi : W(S) \to \mathcal{A}$ uniquely defined by $\pi(x_j) = S_j$ for all $j \in \{1,\ldots, n\}$.  Clearly $\pi$ extends to a homomorphism $\pi : \mathbb{C}\langle S\rangle \to \mathcal{A}$ and thus also extends to a homomorphism $\pi : (\mathbb{C}\langle S\rangle)[[ \{z\}]] \to \mathcal{A}[[ \{z\}]]$ by applying $\pi$ coordinate-wise.
\par 
Let $P \in \mathcal{A}_{\textrm{rat}}[[\{z\} ]]$ be arbitrary.  Using algebraic properties, the proof of \cite{Sa}*{Theorem 2.19(ii)} implies that 
\[
P \in \pi\left(  (\mathbb{C}\langle S\rangle)_{\textrm{rat}}[[ \{z\}]]  \right).
\]
Choose $\overline{P} \in (\mathbb{C}\langle S\rangle)_{\textrm{rat}}[[ \{z\}]] \subseteq (\mathbb{C}(z))_{\textrm{rat}}[[ S]] $ such that $\pi(\overline{P}) = P$.  Recall that 
\[
P_{semi} := \sum_{w \in W(S)} \tau(w(S_1,\ldots, S_n)) w \in \mathbb{C}_{\textrm{alg}}\langle \langle S\rangle \rangle \subseteq (\mathbb{C}(z))_{\textrm{alg}}\langle\langle S\rangle\rangle
\]
by Lemma \ref{certainformalpowerseriesisalgebraic}.  Hence the Haadamard Product
\[
\overline{P} \odot P_{semi} := \sum_{w \in W(S)} \overline{P}(w) P_{semi}(w) w = \sum_{w \in W(S)} \tau(w(S_1,\ldots, S_n)) \overline{P}(w) w
\] 
is an element of $(\mathbb{C}(z))_{\textrm{alg}}\langle\langle S\rangle\rangle$ by a theorem of Sch\"{u}tzenberger from \cite{Sch}.
\par
Since $\overline{P} \odot P_{semi} \in (\mathbb{C}(z))_{\textrm{alg}}\langle\langle S\rangle\rangle$, if we substitute $1 \in \mathbb{C}$ for every element of $S$ we obtain a well-defined power series in $\mathbb{C}[[\{z\}]]$.  Indeed if 
\[
P = \sum_{m\geq 0} p_m(S_1, \ldots, S_n) z^m
\]
for some non-commutative polynomials $p_m$ in $n$ variables, then 
\[
\overline{P} =  \sum_{m\geq 0} (p_m(x_1, \ldots, x_n) + q_m(x_1,\ldots, x_n))z^m 
\]
for some non-commutative polynomials $q_m$ in $n$ variables such that $q_m(S_1, \ldots, S_n) = 0$.  Hence
\[
\overline{P} \odot P_{semi} = \sum_{w \in W(S)} \tau(w(S_1,\ldots, S_n)) \left(\sum_{m\geq 0} (coef(p_m, w) + coef(q_m, w))z^m  \right)w
\]
where $coef(p, w)$ is the element of $\mathbb{C}$ that is the coefficient of $w$ in $p$.  Therefore, by replacing each $w$ with the scalar $1$, we obtain 
\[
\begin{array}{rcl}
 && \sum_{w \in W(S)} \tau(w(S_1,\ldots, S_n)) \left(\sum_{m\geq 0} (coef(p_m, w) + coef(q_m, w))z^m  \right)  \\
 &=&  \sum_{m\geq 0}   \tau\left( \sum_{w \in W(S)} (coef(p_m, w) + coef(q_m, w)) w(S_1,\ldots, S_n)\right) z^m  \\
  &=&  \sum_{m\geq 0}   \tau\left( p_m(S_1, \ldots, S_n) + q_m(S_1, \ldots, S_n) \right) z^m  \\
 &=& \sum_{m\geq 0}\tau(p_m(S_1,\ldots, S_n))z^m = Tr_\mathfrak{M}(P) 
 \end{array} 
\]
as desired.  Thus the proof of \cite{Sa}*{Theorem 2.19(ii)} implies that $Tr_\mathfrak{M}(P)$ is an element of $\mathbb{C}_{\textrm{alg}}[[ \{z\}]]$ as desired.
\end{proof}
With the proof of Theorem \ref{cauchytransformofsemicircularsisalgebraic} complete, we turn our attention to further information that Sauer's results from \cite{Sa} imply.  The main purpose of \cite{Sa} was to show the rationality and positivity of the Novikov-Shubin invariant for matrices with entries in the group algebra of a virtually free group.  In particular, the Novikov-Shubin invariants are well-defined for any finite, tracial von Neumann algebra.
\begin{defn}
\label{NokikovShubinInvariant}
Let $\mathfrak{M}$ be a finite von Neumann algebra with faithful, normal, tracial state $\tau$.  For a positive operator $A \in \mathfrak{M}$ with spectral distribution $F_A$, the Novikov-Shubin invariant $\alpha(A) \in [0,\infty] \cup \{\infty^+\}$ of $A$ is defined as
\[
\alpha(A) : = \left\{
\begin{array}{ll}
\liminf_{t \to 0^+} \frac{\ln(F_A(t) - F_A(0))}{\ln(t)} & \mbox{if } F_A(t) > F_A(0) \mbox{ for all }t >0 \\
\infty^+ & \mbox{otherwise}
\end{array} \right. .
\]
\end{defn}
For a positive operator $A$ in a finite von Neumann algebra $\mathfrak{M}$, it is easy to see that $\alpha(A) = \infty^+$ implies that zero is isolated in the spectrum of $A$.  Furthermore, if $\alpha(A) = \lambda \in [0,\infty)$, then $F_A(t) - F_A(0)$ behaves like $t^\lambda$ as $t$ tends to zero. 
\par
The Novikov-Shubin invariants are of interest in the context of Theorem \ref{cauchytransformofsemicircularsisalgebraic} due to the following result which is directly implied by the proof of \cite{Sa}*{Theorem 3.6}.
\begin{lem}[See \cite{Sa}*{Theorem 3.6} for a proof]
\label{algebraiccauchytransformimpliestheNSinvariantisgood}
Let $\mathfrak{M}$ be a finite von Neumann algebra with faithful, normal, tracial state $\tau$.  Let $A \in \mathfrak{M}$ be a positive operator and let $\mu_A$ is the spectral measure of $A$.  If the Cauchy transform $G_{\mu_A}$ is algebraic, then the Novikov-Shubin invariant $\alpha(A)$ is a non-zero rational number or $\infty^+$.
\end{lem}
The Novikov-Shubin invariants are of interest in terms of determining the decay of the spectral density function at zero due to the following result.
\begin{lem}[See \cite{Luck}*{Theorem 3.14(4)}]
\label{finiteNSimpliesfiniteintegral}
Let $\mathfrak{M}$ be a finite von Neumann algebra with faithful, normal, tracial state $\tau$.  If $A \in \mathfrak{M}$ is a positive operator and $F_A$ is the spectral density function of $A$, then
\[
\lim_{\epsilon \to 0} \int^{\left\|A\right\|}_\epsilon \frac{1}{t}(F_A(t) - F_A(0))\, dt < \infty
\]
provided $\alpha(A) \neq 0$.
\end{lem}
\begin{proof}
If $\alpha(A) = \infty^+$, then $F_A(t) - F_A(0)$ is a right continuous function bounded that is zero on a neighbourhood of zero.  Hence the result follows.  If $\alpha(A) \in (0, \infty]$, then it is trivial to verify from Definition \ref{NokikovShubinInvariant} that there exists a $\delta > 0$ and an $\lambda \in (0, \alpha(A))$ such that $F(t) - F(0) \leq t^\lambda$ for all $0 \leq t \leq \delta$.  Hence
\[
0 \leq \int^\delta_0 \frac{1}{t}(F_A(t) - F_A(0)) \, dt \leq \int^\delta_0 t^{\lambda - 1} \, dt < \infty.
\]
Thus the result follows as $F_A(t) - F_A(0)$ is a right continuous function bounded.
\end{proof}
Furthermore, the following result provides information on how to extract information from the conclusion of Lemma \ref{finiteNSimpliesfiniteintegral} to obtain information about integrating logarithms against the spectral measure.
\begin{lem}[See \cite{Luck}*{Lemma 3.15(1)}]
\label{changefromtraceoflogarithmtoNSinvariant}
Let $\mathfrak{M}$ be a finite von Neumann algebra with faithful, normal, tracial state $\tau$.  If $A \in \mathfrak{M}$ is a positive operator, $F_A$ is the spectral density function of $A$, and $\mu_A$ is the spectral measure of $A$, then
\[
\lim_{\epsilon \to 0} \int^{\left\|A\right\|}_\epsilon \frac{1}{t}(F_A(t) - F_A(0))\, dt < \infty
\]
if and only if 
\[
\lim_{\epsilon \to 0} \int^{\left\|A\right\|}_\epsilon \ln(t) \, d\mu_A(t) > -\infty.
\]
\end{lem}
Combining the above results, we obtain the following.
\begin{thm}
\label{finitetraceoflogarithm}
Let $n,\ell \in \mathbb{N}$, let $X_1, \ldots, X_n$ be freely independent semicircular variables or freely independent Haar unitaries, and let $\mathcal{A}$ be the $*$-algebra generated by $X_1, \ldots, X_n$.  Then 
\[
\lim_{\epsilon \to 0} \int^{\left\|A\right\|}_\epsilon \ln(t) \, d\mu_A(t) > -\infty
\]
for all positive $A \in \mathcal{M}_m(\mathcal{A}) \setminus \{0\}$.  Furthermore, if $\mu_A$ does not have an atom at zero (e.g. when $\ell = 1$ by Theorem \ref{mainapplicationofresult}), then
\[
\int^{\left\|A\right\|}_0 \ln(t) \, d\mu_A(t) > -\infty.
\]
\end{thm}

\end{document}